\documentclass{article}
\usepackage[T1]{fontenc}
\usepackage{times}
\usepackage{rawfonts}
\usepackage{latexsym}
\usepackage{amsmath}
\usepackage{amsfonts}
\usepackage{amssymb}
\usepackage{enumerate}
\usepackage{xspace}
\usepackage{verbatim}
\usepackage{theorem}

\newtheorem{theorem}{Theorem}[section]
\newtheorem{proposition}[theorem]{Proposition}
\newtheorem{lemma}[theorem]{Lemma}
\newtheorem{corollary}[theorem]{Corollary}
{\theorembodyfont{\rm} \newtheorem{assumption}[theorem]{Assumption} }
{\theorembodyfont{\rm} \newtheorem{remark}[theorem]{Remark} }
{\theorembodyfont{\rm} \newtheorem{example}[theorem]{Example} }
\newcommand{\proof}[1][]{{{\sc Proof}\xspace#1. }}

\newcommand{\refeq}[1]{\klasm{\ref{eq:#1}}}

\newcommand{\reza}{ \mathbb{R}\hspace{0.1mm}}
\newcommand{\naza}{ \mathbb{N} }
\newcommand{\varx}{x}
\newcommand{\varxjb}[1][j]{x_{#1-1/2}}
\newcommand{\varxnb}[1][n]{x_{#1-1/2}}

\newcommand{\vary}{y}
\newcommand{\klasm}[1]{(#1)}
\newcommand{\kla}[1]{(#1)}
\newcommand{\klafn}[1]{(#1)}
\newcommand{\noklafn}[1]{#1}
\newcommand{\klami}[1]{(#1)}
\newcommand{\ints}[4]{\int_{#1}^{#2}#3 \, #4}

\newcommand{\cdott}{\hspace{0.4mm}}
\newcommand{\cdottsm}{\hspace{0.3mm}}
\newcommand{\for}{\quad \text{for} \ \ }
\newcommand{\forsm}{\ \text{for} \ }
\newcommand{\intervalarg}[3]{#2 \hspace{0mm} \le \hspace{0mm} #1 \hspace{0mm} \le \hspace{0mm} #3}
\newcommand{\intervalargno}[3]{#2 \le #1 \le #3}

\newcommand{\xmax}{a}
\newcommand{\interval}[2]{\eckkla{#1, #2 }}

\newcommand{\eq}{\ = \ }

\newcommand{\mywlog}{without loss of generality\xspace}

\newcommand{\Ialp}[3][\myalpha]{\klasm{\myi^{{#1}} #2}\klasm{#3}}
\newcommand{\Ialpkla}[3][\myalpha]{\klasm{\myi^{{#1}} \kla{#2}}\klasm{#3}}
\newcommand{\Ialph}[3][\myalpha]{\klasm{\Omega^{{#1}}_h #2}\klasm{#3}}
\newcommand{\mfrac}[2]{\dfrac{\mbox{\footnotesize \raisebox{-0.5mm}{$#1$}}}%
{\mbox{\footnotesize \raisebox{0.8mm}{$#2$}}}}
\newenvironment{mylist}{%
\begin{list}{\tinybullet}
{\setlength{\topsep}{0.2cm}
\setlength{\itemsep}{0mm}
\setlength{\labelwidth}{0mm}
\setlength{\labelsep}{3mm}
\setlength{\itemindent}{3mm}
\setlength{\leftmargin}{0mm}
}}{\end{list}}
\newenvironment{mylist_indent}{%
\begin{list}{\tinybullet}
{\setlength{\topsep}{0.2cm}
\setlength{\itemsep}{0mm}
\setlength{\labelwidth}{2mm}
\setlength{\labelsep}{3mm}
\setlength{\itemindent}{0mm}
\setlength{\leftmargin}{5mm}
}}{\end{list}}
\newcommand{\myi}{\mathcal{V}}
\newcommand{\myalpha}{\alpha}
\newcommand{\alp}{\alpha}
\newcommand{\alpone}{\alpha+1}
\newcommand{\alptwo}{\alpha+2}

\newcommand{\gamalp}{\gamma + \alp}
\newcommand{\malp}{-\alpha}

\newcommand{\malpone}{-\alpha-1}
\newcommand{\malptwo}{-\alpha-2}

\newcommand{\alpmfour}{\alpha-4}

\newcommand{\mon}[1]{y^{#1}}

\newcommand{\halp}{h^\alpha}
\newcommand{\hmalp}{h^{-\alpha}}
\newcommand{\halpone}{h^{\alpone}}

\newcommand{\mysum}[2]{\mathop{\mbox{\small $\dis
\sum\nolimits$}}_{#1}^{#2}}
\newcommand{\myomega}{\omega}
\newcommand{\myomegan}[1][n]{\myomega_{#1}}
\newcommand{\tinybullet}{{\tiny \raisebox{0.6mm}{$ \bullet $}}}
\newcommand{\with}{\quadsm \text{with} \ \ }

\newcommand{\stepsize}{step size\xspace}
\newcommand{\mys}{n}

\newcommand{\myseqq}[3]{#1, \allowbreak #2, \allowbreak \ldots, #3}
\newcommand{\dis}{\displaystyle}
\newcommand{\bn}{\bigskip \noindent}

\newcommand{\ie}{i.\,e.,\xspace}
\newcommand{\nmax}{N}

\newcommand{\bs}{-}
\newcommand{\stepsizes}{\stepsize{}s\xspace}

\newcommand{\eg}{e.g.,\xspace}

\newcommand{\minus}{-}
\newcommand{\wrt}{with respect to\xspace}

\newcommand{\Landau}{\mathcal{O}}

\newcommand{\Landauno}[1]{\Landau\kla{#1}}

\newcommand{\as}{\quad \text{as} \ \ }
\newcommand{\assh}{\ \  \text{as} \ }

\newcommand{\cf}{cf.\mbox{}\xspace}

\newcommand{\ableit}[2]{#1^{\klafn{#2}}}

\newcommand{\myunderbrace}[2]{\underbrace{\raisebox{-0.8mm}{\vphantom{$ #1 $}} #1 }_{\dis #2}}

\newcommand{\ldotsdot}{\ldots \ .}
\newcommand{\lfrac}[2]{#1/ #2}
\newcommand{\myxi}{\xi}
\newcommand{\genfunc}{generating function\xspace}

\newcommand{\insetno}[1]{\{ \, #1 \, \}}
\newcommand{\koza}{ \mathbb{C} }

\newcommand{\modul}[1]{\vert \hspace{0.3mm} #1 \hspace{0.3mm} \vert}
\newcommand{\modulbi}[1]{\big\vert \hspace{0.3mm} #1 \hspace{0.3mm} \big\vert}

\newcommand{\plus}{+}
\newcommand{\maxnorm}[1]{\norm{#1}_\infty }

\newcommand{\norm}[1]{\Vert \hspace{0mm} #1 \hspace{0mm} \Vert}
\newcommand{\Defeq}{\ := \ }
\newcommand{\prim}[1]{#1^{\prime}}

\newcommand{\gridpoint}{grid point\xspace}
\newcommand{\gridpoints}{\gridpoint{}s\xspace}
\newcommand{\mymetmod}[3]{\kla{\widetilde{\Omega}_{#1} \cdottsm #2 } \cdott \klasm{#3}}
\newcommand{\myerrmod}[3]{\kla{\widetilde{E}^\alp_{#1} \cdottsm #2 } \cdott \klasm{#3}}

\newcommand{\myoverbrace}[2]{\overbrace{\raisebox{0.8mm}{\vphantom{$ #1 $}} #1}^{\dis #2}}
\newcommand{\mymetno}[3]{\kla{\Omega_{#1} \cdottsm #2 } \cdott \kla{#3}}
\newcommand{\resp}{respectively\xspace}
\newcommand{\modquamet}{modified quadrature method\xspace}

\newcommand{\feldstretch}[1]{\renewcommand{\arraystretch}{#1}}

\newcommand{\myrnn}[1][\n]{\reza^{#1\times#1}}
\newcommand{\Landausm}[1]{\Landau\klasm{#1}}
\newcommand{\defeq}{:=}
\newenvironment{myenumerate}{%
\begin{list}{(\alph{enumcount})}
{\setcounter{enumcount}{1}\usecounter{enumcount}
\setlength{\topsep}{1mm}
\setlength{\itemsep}{0mm}
\setlength{\labelwidth}{0mm}
\setlength{\labelsep}{1mm}
\setlength{\itemindent}{1mm}
\setlength{\leftmargin}{0mm}
}}{\end{list}}

\newcommand{\klasmsh}[1]{\mbox{\scriptsize\raisebox{0.2mm}{$($}}\nolinebreak\hspace{-0.0mm}\raisebox{-0.1mm}{$#1$}\hspace{-1mm}\nolinebreak\mbox{\scriptsize\raisebox{0.2mm}{$)$}}}
\newcommand{\proofendspruch}[1][]{This completes the proof#1.\proofend}
\newcommand{\proofend}{\endproof}

\def\endproof{$\qquad \Box$}
\newcommand{\rhs}{right\bs{}hand side\xspace}
\newcommand{\xn}[1][n]{x_{#1}}
\newcommand{\xj}[1][\jod]{x_{#1}}
\newcommand{\xjm}[1][\jod]{x_{#1-1}}
\newcommand{\xjb}[1][\jod]{x_{#1-1/2}}
\newcommand{\xjf}[1][\jod]{x_{#1+1/2}}

\newcommand{\myassump}[2]{\begin{assumption} #1 \label{th:#2} \end{assumption}}

\newcommand{\aninv}[1][n]{\omega^{(-1)}_{#1}}
\newcommand{\an}[1][n]{\omega_{#1}}
\newcommand{\rinv}[1]{r^{(-1)}_{#1}}

\newcommand{\schweifklala}[1]{\Big\{\cdott #1 \cdott \Big\} }
\newcommand{\schweifklabi}[1]{\big\{\cdott #1 \cdott \big\} }
\newcommand{\schweifkla}[1]{\{\cdott #1 \cdott \} }

\newcommand{\matvecform}{matrix\bs{}vector formulation\xspace}

\newenvironment{myenumerate_indent}{%
\begin{list}{(\alph{enumcount})}
{\setcounter{enumcount}{1}\usecounter{enumcount}
\setlength{\topsep}{1mm}
\setlength{\itemsep}{-0mm}
\setlength{\labelwidth}{5mm}
\setlength{\labelsep}{2mm}
\setlength{\itemindent}{-0mm}
\setlength{\leftmargin}{7mm}
}}{\end{list}}

\newcommand{\quadti}{\hspace{1.5mm}}

\newcommand{\Ah}{A_h}
\newcommand{\Fh}{F_h^\delta}

\newcommand{\Vh}[1][s]{V_h}

\newcommand{\Dh}{D_h}

\newcommand{\Sh}{S_h}

\newcommand{\myl}{\n}

\newcommand{\gammaalpinv}{\mfrac{1}{\Gamma(\alpha)}}
\newcommand{\gammaalpinvone}{\mfrac{1}{\Gamma(\alpha+1)}}

\newcommand{\gammaalp}{\Gamma(\alpha)}
\newcommand{\gammaalpone}{\Gamma(\alpha+1)}
\newcommand{\gammaalptwo}{\Gamma(\alpha+2)}

\newcommand{\klabi}[1]{\big(\cdottsm #1 \cdottsm\big)}

\newcommand{\quadsm}{\hspace{3mm}}

\newcommand{\gge}{\ > \ }
\newcommand{\omegbar}[2]{\overline{\omega}_{#1#2}}

\newcommand{\klala}[1]{\Big(\cdottsm #1 \cdottsm\Big)}

\newcommand{\n}{n}

\newcommand{\fndelta}[1][n]{f_{#1}^\delta}
\newcommand{\enndelta}[1][j-1/2]{e_{#1}^\delta}
\newcommand{\enndeltamod}[1][j-1/2]{\tilde{e}_{#1}^\delta}

\newcommand{\undelta}[1][n]{u_{#1}^\delta}
\newcommand{\undeltamod}[1][n]{\widetilde{u}_{#1}^\delta}
\newcommand{\undeltab}[1][\mj]{u_{#1-1/2}^\delta}
\newcommand{\undeltabmod}[1][\mj]{\widetilde{u}_{#1-1/2}^\delta}

\newcommand{\genabelinteqspur}{\weaklysingular Volterra integral equations\xspace}
\newcommand{\Landaubi}[1]{\Landau\klabi{#1}}
\newcommand{\Ehdelta}[1][n]{\Delta_h^\delta}
\newcommand{\Fhdelta}[1][n]{F_h^\delta}

\newcommand{\fortwo}{\quad  \text{for} \ \ }

\newcommand{\rhss}{\rhs{}s\xspace}

\newcommand{\mycitea}[2]{#1~\cite{#1[#2]}}
\newcommand{\mynocitea}[2]{\cite{#1[#2]}}
\newcommand{\myciteatwo}[2]{#1~\cite{#2}}

\newcommand{\myciteb}[3]{#1 and #2~\cite{#1_#2[#3]}}
\newcommand{\mycitebtwo}[3]{#1\cdott /\cdott #2~\cite{#3}}

\newcommand{\myk}[2]{k_{#1,#2}}

\newcommand{\fxn}[1][n]{f\klasm{x_{#1}}}

\newcommand{\fps}{power series\xspace}

\newcommand{\kaa}{\ell}
\newcommand{\mysumtxt}[2]{\sum_{#1}^{#2}}

\newcommand{\eckkla}[1]{[\cdott #1 \cdott ]}

\newcommand{\powser}{power series\xspace}

\newcommand{\myfun}{\varphi}
\newcommand{\myfuntil}{\widetilde{\varphi}}

\newcommand{\myfunn}[1][\n]{\myfun_{#1}}
\newcommand{\myfunntil}[1][\n]{\myfuntil_{#1}}

\newcommand{\idstar}{\stackrel{(*)}{ = }}

\newcommand{\cont}{continuous\xspace}

\newcommand{\mydelta}{\chi}

\newcommand{\hdeltatonull}{\kla{h, \delta} \to 0}

\newcommand{\mydeltax}{h}
\newcommand{\mydeltaxalp}{h^\alpha}

\newcommand{\jod}{j}

\newcommand{\N}{N}

\newcommand{\enn}[2]{\klasm{E^{{\myalpha}}_h #1}\klasm{#2}}
\newcommand{\ennmod}[2]{\klasm{\widetilde{E}^{{\myalpha}}_h #1}\klasm{#2}}
\newcommand{\ennsym}{E^{\myalpha}_h}
\newcommand{\rnh}[1][\n]{r_{#1}}
\newcommand{\rh}{R_{h}}
\newcommand{\rha}{S_{h}}
\newcommand{\rhb}{T_{h}}
\newcommand{\rhna}[1][\n]{s_{#1}}
\newcommand{\rhnb}[1][\n]{t_{#1}}
\newcommand{\en}[1][\n]{e_{\n}}

\newcommand{\remarkend}{\quad $ \vartriangle $}
\newcommand{\inset}[1]{\{ \, #1 \, \}}

\newcommand{\intervalargo}[3]{#2 < #1 < #3}

\newcommand{\myt}{\xi}

\newcommand{\infseqzerind}[1]{$ #1_0, #1_1, \ldots $}

\newcommand{\Fps}{Formal power series\xspace}

\newcommand{\infseqind}[1]{$ #1_1, #1_2, \ldots $}

\newcommand{\xiunitdisk}{\myxi \in \koza, \ \modul{\myxi} < 1}
\newcommand{\inverse}{inverse\xspace}
\newcommand{\powinv}[1]{[#1]^{-1}}

\newcommand{\myast}{\cdot}
\newcommand{\octave}{O\textsc{ctave}\xspace}

\newcommand{\C}{F}

\newcommand{\HL}[2]{\C_L^{#1}\interval{0}{#2}}

\newcommand{\HLc}[3]{\C_L^{#1}\interval{#2}{#3}}
\newcommand{\Hsp}[2]{\C^{#1}\interval{0}{#2}}
\newcommand{\Hspc}[3]{\C^{#1}\interval{#2}{#3}}
\newcommand{\myb}{b}

\newcommand{\normlinfomeg}[1]{\Vert #1 \Vert_{\infty,\mytau}}
\newcommand{\normlone}[1]{\Vert #1 \Vert_{1}}
\newcommand{\linfomeg}{\ell_{\mytau}^{\infty}}
\newcommand{\linfomegtil}{\ell_{\wtilde}^{\infty}}
\newcommand{\czeromeg}{c_{\mytau}^{0}}
\newcommand{\czeromegtilde}{c_{\tildeomega}^{0}}
\newcommand{\lone}{\ell^{1}}
\newcommand{\Dr}{\mathcal{D}}

\newcommand{\ainvn}[1][n]{a_{#1}^{(-1)}}
\newcommand{\wtilden}[1][n]{\widetilde{\mytau}_{#1}}
\newcommand{\tildeomega}{\widetilde{\mytau}}
\newcommand{\wtilde}{\widetilde{\mytau}}
\newcommand{\Nbf}{\mathbf{N}}
\newcommand{\mytau}{\sigma}
\newcommand{\mybeta}{\beta}
\newcommand{\taun}{\tau_n}

\newcommand{\tn}[1][n]{\kappa_{#1}}
\newcommand{\myj}{j}
\newcommand{\mj}{n}

\renewcommand{\max}{\mathop{\textup{max}}}
\renewcommand{\sup}{\mathop{\textup{sup}}}

\newcommand{\repmidrule}{product midpoint rule\xspace}
\newcommand{\modrepmidrule}{modified \repmidrule}
\newcommand{\ph}{p_h}
\newcommand{\qh}{q_h}
\newcommand{\myq}{q}
\newcommand{\calp}{c_\alp}
\newcommand{\nth}[1][n]{$#1$\hspace{0.15mm}th\xspace}
\newcommand{\mytp}{p}
\newcommand{\w}[2]{w_{#1#2}}
\newcommand{\weaklysingular}{Abel-type\xspace}
\newenvironment{myenumerate_roman}{%
\begin{list}{(\roman{enumcountroman})}
{\setcounter{enumcountroman}{1}\usecounter{enumcountroman}
\setlength{\topsep}{0.2cm}
\setlength{\itemsep}{0mm}
\setlength{\labelwidth}{0mm}
\setlength{\labelsep}{3mm}
\setlength{\itemindent}{3mm}
\setlength{\leftmargin}{0mm}
}}{\end{list}}
\newcommand{\myfunjb}[1][\varxjb]{\myfun\klasm{#1}}
\newcommand{\myfunpjb}{\prim{\myfun}\klasm{\varxjb}}
\newcommand{\Inthesequel}{In what follows,\  }

\title{The \repmidrule for \weaklysingular Volterra integral equations of the first kind with perturbed data
\date{} }
\newcounter{enumcountroman}

\author{Robert Plato%
\thanks{Department of Mathematics, University of Siegen,
Walter-Flex-Str.~3, 57068 Siegen, Germany.}
}

\numberwithin{equation}{section}

\begin{document}

\maketitle
\newcounter{enumcount}
\renewcommand{\theenumcount}{(\alph{enumcount})}

\bibliographystyle{plain}

\begin{abstract}
In the present paper we consider the regularizing properties
of the \repmidrule for the stable solution of \weaklysingular Volterra integral 
equations of the first kind with perturbed \rhss.
The H\"older continuity of the solution and its derivative is carefully taken into account,
and correction weights are considered to get rid of initial conditions. 
The proof of the inverse stability of the quadrature weights
relies on Banach algebra techniques.
Finally, numerical results are presented.
\end{abstract}

\section{Introduction}
\label{intro-midpointrule}
\subsection{Preliminary remarks}
In this paper we consider linear \genabelinteqspur of the following form,
\begin{eqnarray}
\klasm{Au}\klasm{\varx}  =
\mfrac{1}{\Gamma\klami{\alpha}}
\ints{0}{\varx}
{(\varx-\vary)^{\alp-1} k\kla{\varx,\vary} u\klami{\vary} }
{d \vary}
=  f\klasm{\varx}
\for \intervalarg{\varx}{0}{\xmax},
\label{eq:weaksing-inteq}
\end{eqnarray}
with $ 0 < \alpha < 1 $ 
and $ \xmax > 0 $,
and with a sufficiently smooth kernel function
$ k: \inset{(x,y) \in \reza^2 \ \mid \ 0 \le y \le x \le \xmax } \to \reza $,
and $ \Gamma $ denotes Euler's gamma function.
Moreover, the function $ f: \interval{0}{\xmax} \to \reza $
is supposed to be approximately given,
and a function $ u: \interval{0}{\xmax} \to \reza $
satisfying equation \refeq{weaksing-inteq} is to be determined.

In the sequel we suppose that the kernel function does not vanish on the
diagonal $ 0 \le \varx = \vary \le \xmax $, and
\mywlog we may assume that
\begin{align}
k\klasm{\varx,\varx} = 1 \for
\intervalarg{\varx}{0}{\xmax}
\label{eq:k_eq_one}
\end{align}
holds. 

For the approximate solution of equation \refeq{weaksing-inteq} with an exactly given \rhs $ f $, there exist many quadrature methods,
see \eg \mycitebtwo{Brunner}{van der Houwen}{Brunner_Houwen[86]}, \mycitea{Linz}{85}, 
and \mycitea{Hackbusch}{95}.
One of these methods is the \repmidrule which is considered in detail, \eg in
\myciteb{Weiss}{Anderssen}{72} and in \mycitea{Eggermont}{79}, see also
\cite[Section 10.4]{Linz[85]}.

In the present paper we investiate,
for perturbed \rhss in equation \refeq{weaksing-inteq},
the regularizing properties of the \repmidrule.
The smoothness of the solution is classified in terms of H\"older continuity of the function and its derivative is considered.  We also give a new proof of the inverse stability of the quadrature weights which
relies on Banach algebra techniques and may be of independent interest.
Finally, some numerical illustrations are presented.
\subsection{The Abel integral operator}
As a first step we consider in \refeq{weaksing-inteq}
the special situation $ k \equiv 1 $. On the other hand, for technical reasons we allow arbitrary intervals $ \interval{0}{b} $ with $ 0 < b \le a $ instead of the fixed interval $ \interval{0}{a} $ which allows to extend the obtained results for arbitrary kernels $ k $.

The resulting integral operator is the Abel integral operator
\begin{align}
\Ialp{\myfun}{\varx}
= 
\mfrac{1}{\Gamma\klami{\alpha}}
\ints{0}{\varx} 
{ \klasm{\varx-\vary}^{\alpha-1} \myfun\klami{\vary} }
{d \vary}
\for \intervalarg{\varx}{0}{\myb}, 
\label{eq:ialp_def}
\end{align}
where $ \myfun: \interval{0}{\myb} \to \reza $ is supposed to be a piecewise \cont function. One of the basic properties of the Abel integral operator
is as follows,
\begin{align}
\Ialp{\mon{q}}{x}
=
\tfrac{\Gamma\klafn{q+1}}{\Gamma\klafn{q+1+\alpha}}
\cdott x^{q+\alpha}
\for x \ge 0
\qquad \kla{q \ge 0},
\label{eq:ialp_monom}
\end{align}
where $ \mon{q} $ is short notation for the mapping $ y \mapsto y^q $.
In the sequel, frequently we make use of the following elementary estimate:
\begin{align}
\sup_{0 \le x \le \myb} 
\modul{\Ialp{\myfun}{x}}
\le \mfrac{\myb^\alpha}{\Gamma(\alpha+1)}
\sup_{0 \le x \le \myb} 
\modul{\myfun\klasm{y}} \qquad
\kla{\myfun : \interval{0}{\myb} \to \reza \textup{\ piecewise continuous}}.
\label{eq:ialp_norm}
\end{align}
Other basic properties of the Abel integral operator can be found 
\eg~in \myciteb{Gorenflo}{Vessella}{91}
or \mycitea{Hackbusch}{95}.
\section{The \repmidrule for Abel integrals}
\label{midpointrule-basics}
\subsection{The method}
For the numerical approximation of the Abel integral operator \refeq{ialp_def}
we introduce equidistant \gridpoints 
\begin{align}
\xn = \n \mydeltax, \qquad \n = k/2, \quad k = 0, 1, \ldots,2\N, 
\with \mydeltax = \frac{\xmax}{\N},
\label{eq:grid-points}
\end{align}
where $ \N $ is a positive integer.
For a given \cont function $ \myfun: \interval{0}{\xn} \to \reza
\ (\n \in \inset{\myseqq{1}{2}{\N} }) $, the \repmidrule for the numerical approximation of the Abel
integral $ \Ialp{\myfun}{\xn} $ 
is obtained by replacing
the function $ \myfun $ on each subinterval
$ \interval{\xjm}{\xj}, \ j = 1,2,\ldots,\n $, 
by the constant term $ \myfunjb $, respectively:
\begin{align}  
\Ialp{\myfun}{\xn}
& \approx
\gammaalpinv
\mysum{\jod=1}{\n}
\schweifklala{
\ints{\varx_{\jod-1}}{\varx_\jod} {\klasm{\xn -\vary}^{\alpha-1} }{d \vary}
}
\myfunjb
\label{eq:midpoint-rule-start}
\\
&=
\gammaalpinvone
\mysum{\jod=1}{\n}
\schweifklabi{ \kla{\xn - \xjm}^{\alpha} - \kla{\xn - \xj}^{\alpha} } 
\myfunjb
\nonumber
\\
&=
\mfrac{\halp}{\gammaalpone}
\mysum{\jod=1}{\n}
\schweifklabi{
 \kla{\n - \jod + 1}^{\alpha} \minus \kla{\n - \jod}^{\alpha}
}
\myfunjb
\nonumber
\\[-1mm]
&=
\mydeltaxalp \mysum{\jod=1}{\n} \an[\n-\jod] \myfunjb
=:
\Ialph{\myfun}{\xn},
\label{eq:midpoint-rule}
\end{align} 
where the quadrature weights $ \an[0], \an[1], \ldots $ are given by 
\begin{align} 
\an[s] & =
\gammaalpinvone
\schweifklabi{\klasm{s+1}^{\alp} - s^{\alp}}
\for s = 0, 1, \ldots \ .
\label{eq:omegan-def}
\end{align} 
The weights
have the asymptotic behavior
$ \myomegan = \frac{1}{\gammaalp}  \n^{\alpha-1} \plus \Landauno{\n^{\alpha-2}} $
as $ \n \to \infty $.
\subsection{The integration error -- preparations}
In the sequel, we consider the integration error 
\begin{align} 
\enn{\myfun}{\xn} = \Ialp{\myfun}{\xn} - \Ialph{\myfun}{\xn}
\label{eq:midpoint-rule-error-def}
\end{align} 
under different smoothness assumptions on the function $ \myfun $.
As a preparation, for $  c < d,  L \ge 0, m = 0, 1,
\ldots $ and $ 0 < \beta \le 1 $, we introduce the space $ \HLc{m+\beta}{c}{d} $
of all functions $ \myfun : \interval{c}{d} \to \reza $ that are continuously
differentiable up to order $m$, and the derivative $ \ableit{\myfun}{m} $ of order
$ m $ is H\"older continuous of order $ \beta $ with H\"older constant $ L \ge 0
$, \ie
\begin{align}
\HLc{m+\beta}{c}{d}
= \inset{\varphi \in C^m\interval{c}{d}
\mid 
\modul{\ableit{\varphi}{m}(x) - \ableit{\varphi}{m}(y)} \le L \modul{x-y}^\beta
\for x, y \in \interval{c}{d}}.
\label{eq:hoelder-with-L}
\end{align}
The space of H\"older continuous functions of order $ m + \beta $ on the interval
$ \interval{c}{d} $ is then given by  
\begin{align*}
\Hspc{m+\beta}{c}{d}
= \inset{\varphi: \interval{c}{d} \to \reza
\mid 
\myfun \in \HLc{m+\beta}{c}{d} \text{ for some constant } L \ge 0}.
\end{align*}
Other notations for those spaces are quite common, \eg $ C^{m,\beta}\interval{c}{d} $,
\cf \cite[section 2]{Brunner[04]}.
As a preparation, for $ n \in  \inset{\myseqq{1}{2}{N}} $ and 
$ \myfun: \interval{0}{\xn} \to \reza $
we introduce the corresponding piecewise constant interpolating spline
$ \ph \myfun: \interval{0}{\xn} \to \reza $, \ie
\begin{align}
(\ph \myfun)(\vary) \equiv \myfunjb
\for \xjm \le \vary < \xj \qquad \kla{\jod = 1,2,\ldots,\n},
\label{eq:ph-def}
\end{align}
and in the latter case $ j = \n $, this setting is also valid for $ \vary = \xn $.
For $ \myfun \in \Hsp{\gamma}{\xn} $ with $ 0 < \gamma \le 1 $,
it follows from zero order Taylor expansions at the \gridpoints
that 
\begin{align}
\myfun(y) = (\ph \myfun)(y) + \Landauno{h^\gamma},
\quad 0 \le y \le \xn,
\label{eq:interpol-error-1}
\end{align}
uniformly both on $ \interval{0}{\xn} $ and for $ \myfun \in \HLc{\gamma}{0}{\xn} $, with any 
arbitrary but fixed constant $ L \ge 0 $, and also uniformly for $ n = 1,2,\ldots, \N $.

We consider the smooth case $ \myfun \in C^1\interval{0}{\xn},
\ \n \in \inset{\myseqq{1}{2}{\N} } $, next. 
Let
$ \qh \myfun: \interval{0}{\xn} \to \reza $ be given by
\begin{align}
(\qh \myfun)(\vary) = \myfunjb + (\vary-\xjb)\myfunpjb
\for \xjm \le \vary < \xj \quad \kla{\jod = 1,2,\ldots, \n},
\label{eq:qh-def}
\end{align}
and
in the latter case $ j = \n $, this definition is extended to the case $ \vary = \xn $.
For $ \myfun \in \Hsp{\gamma}{\xn} $ with $ 1 < \gamma \le 2 $,
first order Taylor expansions at the \gridpoints
yield
\begin{align}
\myfun(y) = (\qh \myfun)(y) + \Landauno{h^\gamma},
\quad 0 \le y \le \xn,
\label{eq:interpol-error-2}
\end{align}
uniformly in the same manner as for \refeq{interpol-error-1}.
\subsection{The integration error}
We are now in a position to consider, under different smoothness conditions on the function $ \varphi $, representations for the integration errors $ \enn{\myfun}{\xn} $ introduced in \refeq{midpoint-rule-error-def}.
\begin{lemma}
Let $ n \in  \inset{\myseqq{1}{2}{N}} $, and moreover let 
$ \myfun: \interval{0}{\xn} \to \reza $ be a continuous function.
We have the following representations for the quadrature error 
$ \enn{\myfun}{\xn} $ introduced in \refeq{midpoint-rule-error-def}:
\begin{myenumerate_indent}
\item
We have
\begin{align}
\enn{\myfun}{\xn} = \Ialpkla{\myfun- \ph \myfun}{\xn}.
\label{eq:midpoint_error_0}
\end{align}
\item
For $ \myfun \in C^1\interval{0}{\xn} $
we have 
\begin{align}
\enn{\myfun}{\xn} = 
\halpone\mysum{\jod=1}{\n} \tau_{\n-\jod} \myfunpjb
+ \Ialpkla{\myfun- \qh \myfun}{\xn},
\label{eq:midpoint_error_1}
\end{align}
where
\begin{align} 
\tau_n 
& =
\mfrac{1}{\gammaalptwo}
\schweifkla{\klasm{\n+1}^{\alpone} \minus \n^{\alpone }}
-\mfrac{1}{2\gammaalpone}
\schweifkla{\klasm{\n+1}^{\alp} \plus \n^{\alp}}
\forsm \n = 0, 1, \ldots \ .
\label{eq:interr-beta+1-c}
\end{align} 
\end{myenumerate_indent}
\label{th:midpoint-error}
\end{lemma}
\proof
The error representation \refeq{midpoint_error_0}
is an immediate consequence of the identities \refeq{midpoint-rule-start} and \refeq{midpoint-rule}. 
For the verification of the second error representation \refeq{midpoint_error_1},
we use the decomposition 
\begin{align*} 
\enn{\myfun}{\xn} = \Ialpkla{\myfun- \ph \myfun}{\xn} = 
\Ialpkla{\qh \myfun - \ph \myfun}{\xn} + \Ialpkla{\myfun -\qh \myfun}{\xn}, 
\end{align*} 
and we have to consider the first term on the \rhs in more detail.
Elementary computations show that
\begin{align} 
\gammaalpinv \ints{\xjm}{\xj}{ \klasm{\xn - \vary}^{\alpha-1} \klasm{\vary \minus \varxjb}}
{d y} 
= \halpone \tau_{\n-\jod}
\for \jod = \myseqq{1}{2}{\n}.
\label{eq:interr-beta+1-b}
\end{align}
From \refeq{interr-beta+1-b}, the second error representation \refeq{midpoint_error_1} already follows.
\proofendspruch[ of the lemma]

\bn
A Taylor expansion of the \rhs of \refeq{interr-beta+1-c}
shows that the coefficients $ \tau_{\myl} $ have the following
asymptotic behavior:
\begin{align}
\tau_\myl &= 
\mfrac{1-\alp}{12 \gammaalp} \myl^{\alpha-2} \plus \Landauno{\myl^{\alpha-3} }
\as \myl \to \infty.
\label{eq:taul_asymp}
\end{align}
Lemma \ref{th:midpoint-error} is needed in the proof of our main theorem. It is stated in explicit form here since it immediately becomes clear from this lemma that,
for each $ \myfun \in \Hsp{\gamma}{a} $ with $ 0 < \gamma \le \alpone $,
the interpolation error satisfies
\begin{align*}
\enn{\myfun}{\xn} = \Landauno{h^\gamma} \as h \to 0
\end{align*}
uniformly for $ n = 0, 1, \ldots, \N $. This follows from 
\refeq{interpol-error-1} and \refeq{interpol-error-2}, and from the absolute summability
$ \mysumtxt{n=0}{\infty}\modul{\taun} \allowbreak < \infty $,
\cf \refeq{taul_asymp}.
\section{The \repmidrule for \weaklysingular first-kind Volterra integral equations with perturbations}
\label{quad-error}
\subsection{Some preparations}
We now return to the \weaklysingular integral equation \refeq{weaksing-inteq}. For the numerical approximation we consider this equation at grid points
$ \xn = \n \mydeltax, n = 1, 2, \ldots,\N $ with 
$ \mydeltax = \lfrac{\xmax}{\N} $, \cf\refeq{grid-points}. 
The resulting integrals are approximated by the \repmidrule, respectively, see \refeq{midpoint-rule} with $ \myfun(y) =  k\klasm{\xn,y} u\klasm{y} $ for
$ \intervalargno{y}{0}{\xn} $.

\Inthesequel we suppose that the \rhs of equation \refeq{weaksing-inteq} is only approximately given with
\begin{align}
\modul{ \fndelta - f\klasm{\xn} } \le \delta
\for n = \myseqq{1}{2}{\nmax},
\label{eq:rhs-assump}
\end{align}
where $ \delta > 0 $ is a known noise level. For this setting, the \repmidrule for the numerical solution of equation \refeq{weaksing-inteq} looks as follows:
\begin{align}
\halp \mysum{j=1}{n} \an[n-j] \cdott k\kla{\xn,\varxjb} \cdott \undeltab[j]
=
\fndelta, \qquad n = 1, 2, \ldots, \nmax.
\label{eq:midpoint-noise-rule}
\end{align}
The approximations $ \undeltab \approx u(\varxnb) $ for $ j = 1, 2, \ldots, \nmax $
can be determined recursively by using scheme \refeq{midpoint-noise-rule}.

For the main error estimates, we impose the following conditions.
\myassump{
\begin{myenumerate}
\item
\label{item:assump-u}
There exists a solution $ u: \interval{0}{\xmax} \to \reza $ to the integral equation
\refeq{weaksing-inteq} with $ u \in \Hsp{\gamma}{\xmax} $, where 
$ \calp \defeq \min\{\alp,1-\alp\} < \gamma \le 2 $.

\item
\label{item:assump-k=1}
There holds $ k\klasm{\varx,\varx} = 1 $ for each $ \intervalarg{\varx}{0}{\xmax}
$. 

\item
\label{item:assump-k-smooth}
The kernel function $ k $ has Lipschitz continuous partial derivatives up to the order 2.

\item
The \gridpoints $ \xn $ are given by \refeq{grid-points}.

\item
The values of the \rhs of equation \refeq{weaksing-inteq}
are approximately given at the \gridpoints, \cf\refeq{rhs-assump}.
\end{myenumerate}
}{midpoint-assump}
\subsection{\Fps}
As a preparation for the proof of the main stability result of the present paper, 
\cf{}Theorem \ref{th:main-midpoint},
we next consider \powser. \Inthesequel we identify sequences $ (b_\myl)_{\myl \ge 0} $ of complex numbers with their (formal) \fps
$ b\klasm{\myxi} = \sum_{\myl=0}^{\infty} b_\myl \myxi^\myl $, 
with $ \myxi \in \koza $.
Pointwise multiplication of two \powser
\begin{align*}
\klala{\mysum{\kaa=0}{\infty} b_\kaa \myxi^\kaa}
\cdot
\klala{\mysum{j=0}{\infty} c_j \myxi^j}
=
\mysum{n=0}{\infty} d_n \myxi^n,
\with
d_n \defeq 
\mysum{\kaa=0}{n} b_\kaa c_{n-\kaa}
\fortwo n = 0, 1, \ldots
\end{align*}
makes the set of \fps 
into a complex commutative algebra with unit element $ 1 + 0 \cdot \myxi + 0 \cdot \myxi^2 + \cdots $~.
For any \fps
$ b\klasm{\myxi} = \sum_{\myl=0}^{\infty} b_\myl \myxi^\myl $ 
with $ b_0 \neq 0 $, there exists a
\fps which inverts the \fps  $ b $ with respect to pointwise multiplication, 
and it is denoted by
$ 1/b\klasm{\myxi} $ or by
$ \powinv{b\klasm{\myxi}} $. For a thorough introduction to formal \fps see, 
\eg~\mycitea{Henrici}{74}.

\Inthesequel we consider the \inverse
\begin{align}
\powinv{\myomega\klasm{\myxi}}
=
\mysum{n=0}{\infty} \aninv \myxi^n
\label{eq:ainv-def}
\end{align}
of the \genfunc
$ \myomega\klasm{\myxi} =
\sum_{n=0}^{\infty} \an \cdott \myxi^n $, with $ \an $ as 
in \refeq{omegan-def}.
\begin{lemma}
The coefficients in \refeq{ainv-def} have the following properties:
\begin{align}
& \aninv[0] > 0, 
\qquad
\aninv < 0 \for n = 1,2,\ldots,
\label{eq:omeganinv-negative} \\
& \aninv[0] = \Gamma(\alpone)
= \mysum{n=1}{\infty} \modul{\aninv},
\label{eq:omeganinv-sum} \\
& \aninv = \Landauno{n^{\malpone}} \as n \to \infty.
\label{eq:omeganinv-decay}
\end{align}
\label{th:omeganinv-props}
\end{lemma}
Estimate \refeq{omeganinv-decay} can be found in \mynocitea{Eggermont}{79}.
Another proof of \refeq{omeganinv-decay} which uses Banach algebra theory and may be of independent interest is given in section \ref{abel-stability} of the present paper.
Section \ref{abel-stability} also contains proofs of the other statements in Lemma \ref{th:omeganinv-props}.

Lemma \ref{th:omeganinv-props} is needed in the proof of our main result, \cf Theorem  
\ref{th:main-midpoint} below and section \ref{appendix_b}. 
We state the lemma here in explicit form since it is fundamental in the stability estimates.
\subsection{The main result}
We next present the first main result of this paper, cf.~the following theorem, where
different situations on the smoothness of the  solution $ u $ are considered. For comments on the estimates presented 
in the theorem, see Remark \ref{th:main-midpoint-remark} below.
\begin{theorem}
\label{th:main-midpoint}
Let the conditions of Assumption \ref{th:midpoint-assump} be 
satisfied, and consider the approximations $ \myseqq{\undelta[1/2]}{\undelta[3/2]}{\undeltab[\nmax]} $
determined by scheme \refeq{midpoint-noise-rule}.
Let $ \calp \defeq \min\{\alp,1-\alp\} $.
\begin{myenumerate}
\item 
If $ \calp < \gamma \le 1 +  \calp $, then we have
\begin{align}
\max_{\mj=\myseqq{1}{2}{\nmax}} 
\modul{\undeltab - u\klasm{\varxnb} } 
= \Landauno{h^{\gamma-\calp} + \mfrac{\delta}{\halp}} \as \hdeltatonull.
\label{eq:th-main-midpoint-a}
\end{align}
\item Let $ 2-\alp < \gamma \le 2 $, and in addition let $ u(0) = \prim{u}(0) = 0 $ be satisfied.
Then 
\begin{align}
\max_{\mj=\myseqq{1}{2}{\nmax}} 
\modul{\undeltab - u\klasm{\varxnb} } 
=
\Landauno{h^{\gamma-1+\alp} + \mfrac{\delta}{\halp}} \as \hdeltatonull.
\label{eq:th-main-midpoint-c}
\end{align}
\end{myenumerate}
\end{theorem}
The proof of Theorem \ref{th:main-midpoint} is given in section \ref{appendix_b}. 
Below we give some comments on Theorem \ref{th:main-midpoint}.
\begin{remark}
\label{th:main-midpoint-remark}
\begin{myenumerate}
\item 
\label{it:alp_le_1_2}
In the case $ \alp \le \tfrac{1}{2} $ we have the following estimates: 
\begin{align*}
\max_{\mj=\myseqq{1}{2}{\nmax}} \modul{\undelta - u\klasm{\varxnb} } 
=  \left\{ \begin{array}{rl}
\Landauno{h^{\gamma-\alp} + \mfrac{\delta}{\halp}}, & \textup{if} \  \alp < \gamma \le \alpone, \\
\Landauno{h^{\gamma-1 +\alp} + \mfrac{\delta}{\halp}}, & \textup{if} \  2-\alp < \gamma \le 2, 
\ u(0) = \prim{u}(0) = 0.
\end{array}
\right.
\end{align*}

\item
\label{it:alp_ge_1_2}
In the case we $ \alp \ge \tfrac{1}{2} $ the following estimates hold: 
\begin{align*}
\max_{\mj=\myseqq{1}{2}{\nmax}} \modul{\undelta - u\klasm{\varxnb} } 
= 
\Landauno{h^{\gamma- 1 + \alp} + \mfrac{\delta}{\halp}}, & \textup{ if } \  1-\alp < \gamma \le 2-\alp, \\[-3mm]
& \textup{ or if } \ 2-\alp < \gamma \le 2, \ u(0) = \prim{u}(0) = 0. 
\end{align*}

\item 
\label{it:weiss_eggermont}
The noise-free rates, obtained for $ \gamma = 1 $ and $ \gamma = 2 $, 
basically coincide with those given in the papers by
\myciteb{Weiss}{Anderssen}{72} and by Eggermont \mynocitea{Eggermont}{81}.
\item
\label{it:max_rate}
The maximal rate in the noise-free case $ \delta = 0 $ is
$ \Landauno{h} $ without initial conditions, and it is
obtained for $ \gamma = 1 +\calp $. This rate is indeed maximal, which can be seen by considering
the error at the first \gridpoint $ x_{1/2} $, obtained for the function $ u(y) = y $,
\cf \myciteb{Weiss}{Anderssen}{72}.
Under the additional assumption $ u(0) = \prim{u}(0) = 0 $, the maximal rate
is $ \Landauno{h^{\alpone}} $, obtained for $ \gamma = 2 $. 

\item It is not clear if the presented rates are optimal.
\remarkend
\end{myenumerate}
\end{remark}
\Inthesequel for \stepsizes $ h = \xmax/N $ we write, with a slight abuse of notation, 
$ h \sim \delta^\beta $ as $ \delta \to 0 $, 
if there exist real constants $ c_2 \ge c_1 > 0 $ such that
$ c_1 h \le \delta^{\beta} \le c_2 h $ holds for $ \delta \to 0 $.
As an immediate consequence of Theorem \ref{th:main-midpoint}
we obtain the following main result of this paper.
\begin{corollary}
Let Assumption \ref{th:midpoint-assump} be satisfied.
\begin{mylist_indent}
\item
Let $ \alpha \le 1/2 $ and $ \alpha < \gamma \le \alpone $. For
$ h = h(\delta) \sim \delta^{1/\gamma}$ we have
\begin{align*}
\max_{\mj=1,2, \ldots, \nmax} \modul{\undeltab \minus u\klasm{\varxnb}} 
= \Landauno{\delta^{1-\lfrac{\alp}{\gamma}}}
\as \delta \to 0.
\end{align*}

\item
Let one the following two conditions be satisfied: 
(a) $ \alpha \ge 1/2, \, 1- \alpha < \gamma \le 2-\alp $, or (b)
$ \gamma > 2-\alp, \, u(0) = \prim{u}(0) = 0 $. Then for $ h = h(\delta) \sim \delta^{1/(\gamma-1+2 \alp)} $ we have
\begin{align*}
\max_{\mj=1,2, \ldots, \nmax} \modul{\undeltab \minus u\klasm{\varxnb}} 
= \Landaubi{\delta^{1- \tfrac{\alp}{\noklafn{\gamma-1+2\alp}}}}
\as \delta \to 0.
\end{align*}
\end{mylist_indent}
\label{th:noise-corollary}
\end{corollary}
Note that in the case $ \alpha < \tfrac{1}{2} $, for the class of functions satisfying the initial conditions $ u(0) = \prim{u}(0) = 0 $,
there is a gap for $ \alpone < \gamma \le 2 - \alp $ where no improvement in the rates is obtained, \ie we have piecewise saturation $ \Landauno{\delta^{1-\lfrac{\alp}{\gamma}}} $  for this range of $ \gamma $. This is due to the different techniques used in the proof of Theorem \ref{th:main-midpoint}.

We conclude this section with some more remarks.
\begin{remark}
\begin{myenumerate}
\item
We mention some results on other quadrature schemes for the approximate solution of \weaklysingular integral equations of the first kind. The product trapezoidal method is considered, e.g., in 
\mycitea{Weiss}{72}, \mycitea{Eggermont}{81}, and in
\mynocitea{Plato}{12}. Fractional multistep methods are treated in Lubich
\cite{Lubich[86b], Lubich[87]} 
and in \mynocitea{Plato}{05}.
Backward difference product integration methods are considered in
Cameron and McKee~\cite{Cameron_McKee[84],Cameron_McKee[85]}. 

Galerkin methods for Abel-type integral equations are considered, e.g., in
\mycitea{Eggermont}{81} and in V\"ogeli, Nedaiasl and Sauter~\cite{Voegeli_Nedaiasl_Sauter[72]}.
Some general references are already given in the beginning of this paper.
\item
For other special regularization methods for the approximate solution of
Volterra integral equations of the first kind with perturbed \rhss
and
with possibly algebraic-type weakly singular kernels, see \eg
\mycitea{Bughgeim}{99},
\myciteb{Gorenflo}{Vessella}{91},
and the references therein.
\end{myenumerate}
\label{th:main-remark}
\end{remark}
\begin{remark}
The results of Theorem \ref{th:main-midpoint} and Corollary
\ref{th:noise-corollary} can be extended to linear Volterra integral equations of the first kind with smooth kernels, that is, for $ \alpha = 1 $. The resulting method is in fact the classical midpoint rule, and the main error estimate is as follows: if $ 0 < \gamma \le 2 $, then we have
\begin{align*}
\max_{\mj=\myseqq{1}{2}{\nmax}} 
\modul{\undeltab - u\klasm{\varxnb} } 
= \Landauno{h^{\gamma} + \mfrac{\delta}{h}} \as \hdeltatonull,
\end{align*}
and initial conditions are not required anymore then. The choice $ h = h(\delta) \sim \delta^{1/(\gamma+1)}$ then gives
\begin{align*}
\max_{\mj=1,2, \ldots, \nmax} \modul{\undeltab \minus u\klasm{\varxnb}} 
= \Landauno{\delta^{\gamma/(\gamma+1)}}
\as \delta \to 0.
\end{align*}
The proof follows the lines used in this paper, with a lot of simplifications then. In particular, the inverse stability results derived in section \ref{abel-stability} can be discarded in this case.
We leave the details to the reader and indicate the basic ingredients only: 
we have $ \an = 1 $ and $ \taun = 0 $ for $ n = 0,1,\ldots $ then, and in addition, $ \aninv[0] = 1, \aninv[1] = -1 $, and $ \aninv = 0 $ for $ n = 2,3,\ldots $ holds.
For other results on the regularizing properties of the midpoint rule for solving linear Volterra integral equations of the first kind, see
\mynocitea{Plato}{17} and \mycitea{Kaltenbacher}{10}.
\label{th:alp=1}
\end{remark}
\section{Modified starting weights}
\label{start_weights}
For the \repmidrule \refeq{midpoint-rule}, 
applied to a continuous function $ \myfun: \interval{0}{a} \to \reza $,
and with \gridpoints as in \refeq{grid-points}, with $ 1 \le \n \le \N $ and $ \N \ge 2 $,
we now would like to overcome the conditions 
$   \myfun(0) = \prim{\myfun}(0) = 0 $.
For this purpose we consider the modification
\begin{align}
\mymetmod{h}{\myfun}{\xn} \Defeq
\myoverbrace{\halp \mysum{j=1}{n} \myomega_{n-j} \cdott \myfunjb }%
{\dis  = \mymetno{h}{\myfun}{\xn}}
\plus
\halp \mysum{j=1}{2} \w{n}{j} \cdott \myfunjb
\qquad
\qquad \label{eq:midpoint-mod}
\end{align}
as approximation to the fractional integral $ \Ialp{\myfun}{\xn} $
at the considered grid points $ \xn $, \resp.
See Lubich \cite{Lubich[86b], Lubich[87]}
 and \mynocitea{Plato}{05}
for a similar approach for fractional multistep methods.
In \refeq{midpoint-mod}, 
$ \w{n}{1} $ and $ \w{n}{2} $ 
are correction weights for the starting values that are specified
in the following.
In fact, for each $ n = \myseqq{1}{2}{\nmax} $ the correction weights are chosen
such that the modified \repmidrule \refeq{midpoint-mod} is exact at $ \xn = nh $ for polynomials of degree $ \le 1 $,
\ie
\begin{align}
\mymetmod{h}{\mon{q}}{\xn} \eq
\Ialp{\mon{q}}{\xn}
\for q \eq 0,1.
\label{eq:startweights_ansatz}
\end{align}
\subsection{Computation of the correction weights}
For each $ n = \myseqq{1}{2}{\nmax} $, a reformulation of \refeq{startweights_ansatz} gives the following linear system of two equations for the 
starting weights $ \w{n}{j}, \ j = 1,2 $:
\begin{align*}
\halp \kla{\w{n}{1} + \w{n}{2}} & = \enn{1}{\xn}, \qquad
\halpone \kla{\tfrac{1}{2}\w{n}{1} + \tfrac{3}{2}\w{n}{2}} = \enn{y}{\xn},
\end{align*}
\cf \refeq{midpoint-rule-error-def} for the introduction of $ \ennsym $.
On the other hand we have
\begin{align*}
\enn{1}{\xn} = 0, \qquad
\enn{y}{\xn} = \halpone \mysum{\jod=0}{\n-1} \tau_{\jod}.
\end{align*}
Those identities follow from representations
\refeq{midpoint_error_0} and \refeq{midpoint_error_1}, respectively.
From this we obtain
\begin{align}
-\w{n}{1} = \w{n}{2} = 
\mysum{\jod=0}{\n-1} \tau_{\jod}.
\label{eq:wnj-rep}
\end{align}
This in particular means that the correction weights are independent of $ h $. 
We finally note that the asymptotic behavior  
of the coefficients $ \tau_{\jod} $,
\cf \refeq{taul_asymp}, implies
\begin{align}
\w{n}{j} = \Landauno{1}
\as n \to \infty \qquad
\text{for} \quad j = 1,2.
\label{eq:wnj-stable}
\end{align}
\subsection{Integration error of the \modquamet}
We now consider, for each $ n = \myseqq{1}{2}{\nmax} $, the error of the \modrepmidrule,
\begin{align} 
\ennmod{\myfun}{\xn} = \Ialp{\myfun}{\xn} - \mymetmod{h}{\myfun}{\xn},
\label{eq:midpoint-rule-error-def-mod}
\end{align} 
where $ \myfun: \interval{0}{a} \to \reza $ denotes a continuous function.
\begin{lemma}
Let $ n \in  \inset{\myseqq{1}{2}{N}} $, and moreover let 
$ \myfun \in \Hsp{\gamma}{a} $, with $ 0 < \gamma \le 2 $.
We have the following representations of the modified quadrature error 
$ \ennmod{\myfun}{\xn} $ introduced in \refeq{midpoint-rule-error-def-mod}:
\begin{myenumerate_indent}
\item
In the case $  0 < \gamma \le 1 $ we have
$ \ennmod{\myfun}{\xn} = \enn{\myfun}{\xn} 
+ \Landauno{h^{\gamma + \alp}} $ as $ h \to 0 $.

\item
In the case $ 1 < \gamma \le 2 $ we have, with 
$ \myfuntil(y) \defeq \myfun(y) - \myfun(0)  - \prim{\myfun}(0) y $ 
for $ 0 \le y \le \xmax $,
\begin{align*} 
\ennmod{\myfun}{\xn} = \enn{\myfuntil}{\xn} 
+ \Landauno{h^{\gamma + \alp}} \as h \to 0.
\end{align*}
\end{myenumerate_indent}
Both statements hold uniformly for $ n = 1,2,\ldots, N $,
and for $ \myfun \in \HLc{\gamma}{0}{a} $, with any 
constant $ L \ge 0 $.
\label{th:midpoint-error-mod}
\end{lemma}
\proof
\begin{myenumerate}
\item This follows immediately from \refeq{midpoint-mod} and
\refeq{wnj-rep}--\refeq{midpoint-rule-error-def-mod}:
\begin{align*}
\ennmod{\myfun}{\xn} = 
\enn{\myfun}{\xn} + \halp \w{n}{1} \klabi{\myfunjb[x_{3/2}] - \myfunjb[x_{1/2}]}
= \enn{\myfun}{\xn} + \Landauno{h^{\gamma + \alp}} \as h \to 0.
\end{align*}
\item
Using the notation $ \mytp(y) \defeq \myfun(0) + \prim{\myfun}(0) y $, we have
$ \myfun = \myfuntil + \mytp $, and the linearity of the modified error functional gives
\begin{align*}
\ennmod{\myfun}{\xn} & = 
\ennmod{\myfuntil}{\xn} + \myoverbrace{\myerrmod{h}{\mytp}{\xn}}{=0} 
=
\enn{\myfuntil}{\xn} 
\minus
\halp \mysum{j=1}{2} \w{nj} \cdott \myfuntil\kla{\xjb} 
\\[-1mm]
& = 
\enn{\myfuntil}{\xn} 
+ \Landauno{h^{\gamma + \alp }},
\end{align*}
where 
$ \myfuntil(y) = \Landauno{y^\gamma} $ as $ y \to 0 $ has been used,
and the boundedness of the correction weights, \cf \refeq{wnj-stable}, is also taken into account.
\proofend
\end{myenumerate}
\subsection{Application to the Abel-type first kind integral equation }
In what follows, the modified \repmidrule \refeq{midpoint-mod} is applied to numerically solve the algebraic-type weakly singular integral equation \refeq{weaksing-inteq},
with noisy data as in \refeq{rhs-assump}. 
In order to make the starting procedure applicable, in the sequel we assume that the kernel $ k $ can be smoothly extended beyond the triangle $ \inset { 0 \le y \le x \le \xmax } $. For simplicity we assume that the kernel is defined on the whole square.
\myassump{
The kernel function $ k $ has Lipschitz continuous partial derivatives up to the order 2
on $ \interval{0}{\xmax} \times \interval{0}{\xmax} $.}
{kernel-square-cont}
For each $ n = 1, 2, \ldots, \nmax $, we consider the modified \repmidrule \refeq{midpoint-mod} with $ \myfun(y) =  k\klasm{\xn,y} u\klasm{y} $ for $ \intervalargno{y}{0}{a}, \, n = 1, 2, \ldots, \nmax $. This results in the following modified scheme: 
\begin{align}
\halp \mysum{j=1}{n} \an[n-j] k\kla{\xn,\varxjb} \undeltabmod[j]
+ \halp \mysum{j=1}{2} \w{n}{j} k\kla{\xn,\varxjb} \undeltabmod[j]
=
\fndelta, \quad n = 1, 2, \ldots, \nmax.
\label{eq:midpoint-noise-rule-mod}
\end{align}
This scheme can be realized
by first solving a linear system of two equations for the approximations 
$ \undeltabmod[n] \approx u(\varxnb), \ n = 1, 2 $.
The approximations $ \undeltabmod[n] \approx u(\varxnb) $ for $ n = 3, 4, \ldots, \nmax $
can be determined recursively by using scheme \refeq{midpoint-noise-rule-mod} then.
\subsection{Uniqueness, existence and approximation properties
of the starting values}
\label{integsolvec}
We next consider uniqueness, existence and the approximation
properties of the two starting values $ \undeltamod[1/2] $ and $ \undeltamod[3/2] $.
They in fact satisfy the linear system of equations
\begin{align}
\halp \mysum{j=1}{2}
\kla{\myunderbrace{\myomega_{n-j} + \w{n}{j}}{\dis =: \ \omegbar{n}{j} }}
k\klasm{\xn,\xjb} \undeltabmod
=
\fndelta
\for
n = 1, 2,
\label{eq:midpoint_mod_eqsolve-start}
\end{align}
with the notation $ \myomega_{-1} = 0 $.
In matrix notation this linear system of equations 
can be written as
\begin{align}
& & \halp
\myoverbrace{
\left( \begin{array}{@{\cdott}c@{\quad }c@{\cdott}}
\omegbar{1}{1} \cdott k\klasm{x_1,x_{1/2}} &
\omegbar{1}{2} \cdott k\klasm{x_1,x_{3/2}} 
\\ [6mm]
\omegbar{2}{1} \cdott k\klasm{x_2,x_{1/2}} &
\omegbar{2}{2} \cdott k\klasm{x_2,x_{3/2}}
\end{array} \right)
}{=\Sh}
\left( \begin{array}{@{\ }c@{\ }}
\undeltamod[1/2] \\[4mm] \undeltamod[3/2] 
\end{array} \right)
=
\left( \begin{array}{@{\ }c@{\ }}
\fndelta[1] \\[4mm]
\fndelta[2]
\end{array} \right).
\label{eq:S-regular}
\end{align}
\begin{lemma}
The matrix $ \Sh \in \myrnn[2] $ in \refeq{S-regular} is regular
for sufficiently small values of $ h $, and
$ \maxnorm{\Sh^{-1}} = \Landausm{1} $ as $ h \to 0 $, where
$ \maxnorm{\cdot} $ denotes the matrix norm induced by the maximum vector norm on $ \reza^2 $. 
\label{th:S-regular}
\end{lemma}
\proof
We first consider the situation $ k \equiv 1 $ and denote the matrix $ \Sh $ by 
$ T $ in this special case.
From 
\refeq{ialp_monom}
and
\refeq{startweights_ansatz}
it follows
\begin{align*}
\omegbar{n}{1} + \omegbar{n}{2} = \mfrac{n^\alp}{\gammaalpone}
\qquad
\tfrac{1}{2} \omegbar{n}{1} + \tfrac{3}{2}\omegbar{n}{2} = \mfrac{n^{\alpone}}{\gammaalptwo},
\quad n = 1,2.
\end{align*}
Hence the matrix $ T $ is regular and does not depend on $ h $.

We next consider the general case for $ k $. Since 
$ k\klasm{x,x} = 1 $, we have $ k\klasm{x_n,x_m} \to 1 $ as $ h \to 0 $ uniformly for the four values of $ k $ considered
in the matrix $ \Sh $. This shows $ \Sh = T + \Delta $ with $ \maxnorm{\Delta}
\to 0 $ as $ h \to 0 $
so that the matrix $ \Sh $ is regular
for sufficiently small values $ h $,
with $ \maxnorm{\Sh^{-1}} $ being bounded as $ h \to 0 $.
\proofendspruch[ of the lemma]

\bn
We next consider the error of the modified \repmidrule at the first two grid
points $ x_{1/2} $ and $ x_{3/2} $.
\begin{proposition}
Let the conditions of Assumption \ref{th:midpoint-assump} and
Assumption \ref{th:kernel-square-cont} be 
satisfied.
Consider the approximations
$ \undeltamod[1/2] $ and $ {\undeltamod[3/2]} $
determined by scheme \refeq{midpoint-noise-rule-mod} for $ n = 1,2  $.
Then we have
\begin{align*}
\max_{\mj=1,2} 
\modul{\undeltabmod - u\klasm{\varxnb} } 
= \Landauno{h^{\gamma} + \mfrac{\delta}{\halp}} \as \hdeltatonull.
\end{align*}
\label{th:midpoint-mod-starting-values}
\end{proposition}
\begin{proof} 
From \refeq{midpoint-mod}, \refeq{midpoint-rule-error-def-mod}
and Lemma \ref{th:midpoint-error-mod},
applied with 
$ \myfun_n(y) =  k\klasm{\xn,y} u\klasm{y} $ for $ 0 \le y \le a $,
we obtain the representation
\begin{align*}
\halp \mysum{j=1}{2} \omegbar{n}{j} \cdott k\klasm{\xn,\xn[j-1/2]} \cdott
\enndeltamod[j-1/2]
= \ennmod{\myfun_n}{\xn}
+ \fndelta - f(\xn)
= \Landauno{h^{\gamalp} + \delta} 
\for n = 1, 2
\end{align*}
as $ \hdeltatonull $, where $ \enndeltamod = \undeltamod[j-1/2] -
u\klasm{x_{j-1/2}}, \  j = 1,2 $, and the weights 
$ \omegbar{n}{j} $ are introduced in \refeq{midpoint_mod_eqsolve-start}.
Note that
Lemma \ref{th:midpoint-error} and Lemma \ref{th:midpoint-error-mod} 
imply, for the two integers $ n = 1, 2 $, that
$ \ennmod{\myfun_n}{\xn} = \Landauno{h^{\gamma+ \alp}} $ as $ h \to 0 $. 
The proposition now follows from Lemma \ref{th:S-regular}.
\end{proof} 
\subsection{The regularizing properties of the modified scheme}
\begin{theorem}
Let the conditions of Assumption \ref{th:midpoint-assump} and
Assumption \ref{th:kernel-square-cont} be satisfied.
\begin{myenumerate}
\item 
In the case $ \alp \le 1/2 $ we have
\begin{align*}
\max_{\mj=\myseqq{1}{2}{\nmax}} 
\modul{\undeltamod[\mj-1/2] - u\klasm{\varxnb} } 
=  \left\{ \begin{array}{rl}
\Landauno{h^{\gamma-\alp}  + \mfrac{\delta}{\halp}} 
& \textup{if} \  \alp < \gamma \le \alpone, \\[1mm]
\Landauno{h^{\gamma-1 +\alp} + \mfrac{\delta}{\halp}} & \textup{if} \  2-\alp < \gamma \le 2. 
\end{array}
\right.
\end{align*}

\item 
In the case $ \alp \ge 1/2, \, 1-\alpha < \gamma \le 2 $ we have
\begin{align*}
\max_{\mj=\myseqq{1}{2}{\nmax}} 
\modul{\undeltamod[\mj-1/2] - u\klasm{\varxnb} } 
= \Landauno{h^{\gamma-1+\alp} + \mfrac{\delta}{\halp}} \as \hdeltatonull.
\end{align*}
\end{myenumerate}
\label{th:main-midpoint-mod}
\end{theorem}
\begin{proof} 
Let $ \enndeltamod = \undeltamod[j-1/2] - u\klasm{x_{j-1/2}} $ for $ j = 1,2, \ldots, \N $.
From \refeq{midpoint-mod}, \refeq{wnj-stable}, \refeq{midpoint-rule-error-def-mod},
Lemma \ref{th:midpoint-error-mod} and Proposition \ref{th:midpoint-mod-starting-values}
we obtain the representation
\begin{align*}
\halp \mysum{j=1}{n} \omega_{n-j} k\klasm{\xn,\xn[j-1/2]}
\enndeltamod[j-1/2]
& =
\ennmod{\myfun_n}{\xn}
+ f(\xn) - \fndelta
- \halp \mysum{j=1}{2} \w{n}{j} k\klasm{\xn,\xn[j-1/2]} \enndeltamod[j-1/2]
\\
&= 
\ennmod{\myfunn}{\xn}
+ \Landauno{h^{\gamalp} + \delta} 
= \enn{\myfunntil}{\xn}
+ \Landauno{h^{\gamalp} + \delta} 
\end{align*}
as $ \hdeltatonull $, uniformly for $ n = \myseqq{1}{2}{\nmax} $,
where $ \myfunntil = \myfunn $, if $ \gamma \le 1 $, and 
$ \myfunntil(y) = \myfunn(y) - \myfunn(0) - \prim{\myfun}_n(0)y $ for $ \gamma > 1 $.
The theorem now follows by performing the same steps as in the proof of 
Theorem~\ref{th:main-midpoint}.
\end{proof}

\bn
As an immediate consequence of Theorem \ref{th:main-midpoint-mod},
we can derive regularizing properties of the modified scheme.
\begin{corollary}
Let both Assumption \ref{th:midpoint-assump} and Assumption \ref{th:kernel-square-cont} be satisfied.
\begin{mylist_indent}
\item
If $ \alpha \le 1/2 $ and $ \alpha < \gamma \le \alpone $, then
choose $ h = h(\delta) \sim \delta^{1/\gamma}$. The resulting error estimate is
\begin{align*}
\max_{\mj=1,2, \ldots, \nmax} \modul{\undeltabmod \minus u\klasm{\varxnb}} 
= \Landauno{\delta^{1-\lfrac{\alp}{\gamma}}}
\as \delta \to 0.
\end{align*}

\item
Let one the following two conditions be satisfied:
(a) $ \alpha \ge 1/2, \, 1- \alpha < \gamma \le 2-\alp $, or (b)
$ 2-\alp < \gamma \le 2 $.
For $ h = h(\delta) \sim \delta^{1/(\gamma-1+2 \alp)} $ we then have
\begin{align*}
\max_{\mj=1,2, \ldots, \nmax} \modul{\undeltabmod \minus u\klasm{\varxnb}} 
= \Landauno{\delta^{1- \tfrac{\alp}{\noklafn{\gamma-1+2\alp}}}}
\as \delta \to 0.
\end{align*}
\end{mylist_indent}
\label{th:noise-corollary-mod}
\end{corollary}
\section{Numerical experiments}
\label{num_exps}
We next present results of some numerical experiments with the linear \weaklysingular Volterra integral equation 
of the first kind \refeq{weaksing-inteq}. The following example is considered
(for different values of $ 0 < \alpha < 1 $ and $ 0 < \myq \le 2 $):
\begin{align}
k(x,y) = \mfrac{1+xy}{1+x^2}, \qquad
f(x) = \mfrac{1}{\Gamma(q+2+\alp)}\mfrac{x^{\myq+\alp}}{1+x^2} \kla{\myq + 1+\alp + (\myq+1)x^2 }
\for \intervalarg{x,y}{0}{1},
\label{eq:numeric-a}
\end{align}
with exact solution
\kla{\cf\refeq{ialp_monom}}
\begin{align}
u\klasm{y} \eq \tfrac{1}{\Gamma(q+1)} y^\myq
\for \intervalarg{y}{0}{1},
\label{eq:numeric-b}
\end{align}
so that the conditions in \ref{item:assump-u}--\ref{item:assump-k-smooth}
of Assumption \ref{th:midpoint-assump} are satisfied with $ \gamma = q $.
We present experiments for different values of $ \alpha $ and $ q $, sometimes with corrections weights, sometimes without, in order to cover all variants in Corollaries \ref{th:noise-corollary} and \ref{th:noise-corollary-mod}.
Here are additional remarks on the numerical tests.
\begin{mylist_indent}
\item
Numerical experiments with \stepsizes
$ h = 1/2^m $ for $ m = \myseqq{5}{6}{11} $
are employed, \resp.

\item
For each considered \stepsize $ h $, we consider the noise level
$ \delta = \delta(h) = c h^{p + \alp} $, where $ c = 0.3 $, and
$ \Landauno{h^p} $ is the rate for exact data, supplied by 
Theorems \ref{th:main-midpoint} and \ref{th:main-midpoint-mod},
with $ p =  p(\alpha,\myq) $.
The expected error is then of the form $  \max_{n} \modul{\undelta - u\klasm{\xn}} = \Landauno{h^p} =  \Landauno{\delta^{p/(p+\alpha)}} $ as $ h \to 0 $.

\item
In the numerical experiments, the perturbations are of the 
form $ \fndelta = \fxn + \Delta_n $
with uniformly distributed random values $ \Delta_n $ with
$ \modul{\Delta_n} \le \delta $.

\item In all tables, $ \maxnorm{f} $ denotes the maximum norm of the function $ f $.

\item
Experiments are employed using the program system \octave \kla{http://www.octave.org}.
\end{mylist_indent}
\begin{example}
\label{th:example1}
We first consider the situation \refeq{numeric-a}--\refeq{numeric-b},
with $ \alp = \tfrac{1}{2} $ and $ \myq = 2 $. The conditions in \ref{item:assump-u}--\ref{item:assump-k-smooth} of Assumption \ref{th:midpoint-assump} are satisfied with $ \gamma = 2 $ (also for any $ \gamma > 2 $ in fact, but then we have saturation).
We have $ u(0) = \prim{u}(0) = 0 $, so correction weights are not required here. The expected error estimate, with the choice of $ \delta = \delta(h) $ considered in the beginning of this section, is 
$ \max_{n} \modul{\undelta - u\klasm{\xn}} = \Landauno{\delta^{3/4}} =  \Landauno{h^{3/2}} $. The numerical results are shown in Table \ref{tab:num1}.
\begin{table}
\hfill
\begin{tabular}{|| r | c |@{\hspace{5mm} } l | c | c ||} 
\hline
\hline
$ N $  
& $ \delta $
& $ 100 \myast \delta/\maxnorm{f} $
& $ \ \max_{n} \modul{\undelta - u\klasm{\xn}} \ $
& $ \ \max_{n} \modul{\undelta - u\klasm{\xn}} \ / \delta^{3/4} \ $
\\ \hline \hline
 $  32$ & $2.9 \myast 10^{-4}$ & $9.74 \myast 10^{-2}$ & $2.84 \myast 10^{-3}$ & $1.27$   \\
 $  64$ & $7.3 \myast 10^{-5}$ & $2.43 \myast 10^{-2}$ & $1.12 \myast 10^{-3}$ & $1.41$   \\
 $ 128$ & $1.8 \myast 10^{-5}$ & $6.09 \myast 10^{-3}$ & $3.77 \myast 10^{-4}$ & $1.35$   \\
 $ 256$ & $4.6 \myast 10^{-6}$ & $1.52 \myast 10^{-3}$ & $1.37 \myast 10^{-4}$ & $1.38$   \\
 $ 512$ & $1.1 \myast 10^{-6}$ & $3.80 \myast 10^{-4}$ & $5.20 \myast 10^{-5}$ & $1.48$   \\
 $1024$ & $2.9 \myast 10^{-7}$ & $9.51 \myast 10^{-5}$ & $1.89 \myast 10^{-5}$ & $1.53$   \\
 $2048$ & $7.2 \myast 10^{-8}$ & $2.38 \myast 10^{-5}$ & $6.55 \myast 10^{-6}$ & $1.50$   \\
\hline
\hline
\end{tabular}
\hfill 
\caption{Numerical results for Example \ref{th:example1}}
\label{tab:num1}
\end{table} 
\end{example}
\begin{example}
\label{th:example2}
We next consider the situation \refeq{numeric-a}--\refeq{numeric-b},
with $ \alp = 0.9 $ and $ \myq = 0.4 $. The conditions in \ref{item:assump-u}--\ref{item:assump-k-smooth} of Assumption \ref{th:midpoint-assump} are satisfied with $ \gamma = 0.4 $. Since $ \gamma \le 1 $, correction weights are not needed here. The expected error estimate, with $ \delta = \delta(h) $ as in the beginning of this section, is 
$ \max_{n} \modul{\undelta - u\klasm{\xn}} = \Landauno{\delta^{1/4}} = \Landauno{h^{0.3}} $.
The numerical results are shown in Table \ref{tab:num2}.
\begin{table}
\hfill
\begin{tabular}{|| r | c |@{\hspace{5mm} } l | c | c ||} 
\hline
\hline
$ N $  
& $ \delta $
& $ 100 \myast \delta/\maxnorm{f} $
& $ \ \max_{n} \modul{\undelta - u\klasm{\xn}} \ $
& $ \ \max_{n} \modul{\undelta - u\klasm{\xn}} \ / \delta^{1/4} \ $
\\ \hline \hline
 $  32$ & $4.7 \myast 10^{-3}$ & $6.80 \myast 10^{-1}$ & $1.88 \myast 10^{-1}$ & $0.72$   \\
 $  64$ & $2.0 \myast 10^{-3}$ & $2.96 \myast 10^{-1}$ & $1.32 \myast 10^{-1}$ & $0.62$   \\
 $ 128$ & $8.9 \myast 10^{-4}$ & $1.29 \myast 10^{-1}$ & $1.23 \myast 10^{-1}$ & $0.71$   \\
 $ 256$ & $3.9 \myast 10^{-4}$ & $5.61 \myast 10^{-2}$ & $9.61 \myast 10^{-2}$ & $0.69$   \\
 $ 512$ & $1.7 \myast 10^{-4}$ & $2.44 \myast 10^{-2}$ & $8.12 \myast 10^{-2}$ & $0.71$   \\
 $1024$ & $7.3 \myast 10^{-5}$ & $1.06 \myast 10^{-2}$ & $6.77 \myast 10^{-2}$ & $0.73$   \\
 $2048$ & $3.2 \myast 10^{-5}$ & $4.62 \myast 10^{-3}$ & $5.43 \myast 10^{-2}$ & $0.72$   \\
\hline
\hline
\end{tabular}
\hfill 
\caption{Numerical results for Example \ref{th:example2}}
\label{tab:num2}
\end{table} 
\end{example}
\begin{example}
\label{th:example3}
We next consider the situation \refeq{numeric-a}--\refeq{numeric-b},
with $ \alp = 0.2 $ and $ \myq = 0.5 $.
The conditions in \ref{item:assump-u}--\ref{item:assump-k-smooth} of Assumption \ref{th:midpoint-assump} are satisfied with $ \gamma = 0.5 $ then, 
and the expected error estimate
is
$ \max_{n} \modul{\undelta - u\klasm{\xn}} = \Landauno{\delta^{0.6}} = \Landauno{h^{0.3}} $.
The numerical results are shown in Table \ref{tab:num3}.
\begin{table}
\hfill
\begin{tabular}{|| r | c |@{\hspace{5mm} } l | c | c ||} 
\hline
\hline
$ N $  
& $ \delta $
& $ 100 \myast \delta/\maxnorm{f} $
& $ \ \max_{n} \modul{\undelta - u\klasm{\xn}} \ $
& $ \ \max_{n} \modul{\undelta - u\klasm{\xn}} \ / \delta^{0.6} \ $
\\ \hline \hline
$  32$ & $5.3 \myast 10^{-2}$ & $5.12 \myast 10^{0}$ & $1.18 \myast 10^{-1}$ & $0.69$ \\
 $  64$ & $3.8 \myast 10^{-2}$ & $3.62 \myast 10^{0}$ & $8.52 \myast 10^{-2}$ & $0.61$ \\
 $ 128$ & $2.7 \myast 10^{-2}$ & $2.56 \myast 10^{0}$ & $7.78 \myast 10^{-2}$ & $0.69$ \\
 $ 256$ & $1.9 \myast 10^{-2}$ & $1.81 \myast 10^{0}$ & $5.89 \myast 10^{-2}$ & $0.64$ \\
 $ 512$ & $1.3 \myast 10^{-2}$ & $1.28 \myast 10^{0}$ & $5.19 \myast 10^{-2}$ & $0.69$ \\
 $1024$ & $9.4 \myast 10^{-3}$ & $9.05 \myast 10^{-1}$ & $4.20 \myast 10^{-2}$ & $0.69$   \\
 $2048$ & $6.6 \myast 10^{-3}$ & $6.40 \myast 10^{-1}$ & $3.33 \myast 10^{-2}$ & $0.68$   \\
\hline
\hline
\end{tabular}
\hfill 
\caption{Numerical results for Example \ref{th:example3}}
\label{tab:num3}
\end{table} 
\end{example}
\begin{example}
\label{th:example4}
We finally consider the situation \refeq{numeric-a}--\refeq{numeric-b},
with $ \alp = 0.5 $ and $ \myq = 1 $.
Then the conditions in \ref{item:assump-u}--\ref{item:assump-k-smooth} of Assumption \ref{th:midpoint-assump} are satisfied with any $ \gamma > 0 $, and initial conditions are not satisfied in this case.
The presented theory for the \repmidrule without correction weights suggests
that we have $ \max_{n} \modul{\undelta - u\klasm{\xn}} = \Landauno{\delta^{2/3}} = \Landauno{h} $.
The corresponding numerical results are shown in Table \ref{tab:num4}.
\begin{table}
\hfill
\begin{tabular}{|| r | c |@{\hspace{5mm} } l | c | c ||} 
\hline
\hline
$ N $  
& $ \delta $
& $ 100 \myast \delta/\maxnorm{f} $
& $ \ \max_{n} \modul{\undelta - u\klasm{\xn}} \ $
& $ \ \max_{n} \modul{\undelta - u\klasm{\xn}} \ / \delta^{2/3} \ $
\\ \hline \hline
 $  32$ & $1.7 \myast 10^{-3}$ & $2.20 \myast 10^{-1}$ & $1.26 \myast 10^{-2}$ & $0.90$   \\
 $  64$ & $5.9 \myast 10^{-4}$ & $7.79 \myast 10^{-2}$ & $6.47 \myast 10^{-3}$ & $0.92$   \\
 $ 128$ & $2.1 \myast 10^{-4}$ & $2.75 \myast 10^{-2}$ & $3.27 \myast 10^{-3}$ & $0.94$   \\
 $ 256$ & $7.3 \myast 10^{-5}$ & $9.74 \myast 10^{-3}$ & $1.57 \myast 10^{-3}$ & $0.89$   \\
 $ 512$ & $2.6 \myast 10^{-5}$ & $3.44 \myast 10^{-3}$ & $7.72 \myast 10^{-4}$ & $0.88$   \\
 $1024$ & $9.2 \myast 10^{-6}$ & $1.22 \myast 10^{-3}$ & $3.95 \myast 10^{-4}$ & $0.90$   \\
 $2048$ & $3.2 \myast 10^{-6}$ & $4.30 \myast 10^{-4}$ & $2.06 \myast 10^{-4}$ & $0.94$   \\
\hline
\hline
\end{tabular}
\hfill 
\caption{Numerical results for Example \ref{th:example4}, without correction weights}
\label{tab:num4}
\end{table} 

We also consider the \modrepmidrule for the same problem, \ie correction weights are used this time. The presented theory then yields
$ \max_{n} \modul{\undelta - u\klasm{\xn}} = \Landauno{\delta^{3/4}} = \Landauno{h^{3/2}} $.
The related numerical results are shown in Table \ref{tab:num5}.
\begin{table}
\hfill
\begin{tabular}{|| r | c |@{\hspace{5mm} } l | c | c ||} 
\hline
\hline
$ N $  
& $ \delta $
& $ 100 \myast \delta/\maxnorm{f} $
& $ \ \max_{n} \modul{\undelta - u\klasm{\xn}} \ $
& $ \ \max_{n} \modul{\undelta - u\klasm{\xn}} \ / \delta^{3/4} \ $
\\ \hline \hline
 $  32$ & $2.9 \myast 10^{-4}$ & $3.89 \myast 10^{-2}$ & $2.10 \myast 10^{-3}$ & $0.94$   \\
 $  64$ & $7.3 \myast 10^{-5}$ & $9.74 \myast 10^{-3}$ & $6.56 \myast 10^{-4}$ & $0.83$   \\
 $ 128$ & $1.8 \myast 10^{-5}$ & $2.43 \myast 10^{-3}$ & $2.88 \myast 10^{-4}$ & $1.03$   \\
 $ 256$ & $4.6 \myast 10^{-6}$ & $6.09 \myast 10^{-4}$ & $8.66 \myast 10^{-5}$ & $0.87$   \\
 $ 512$ & $1.1 \myast 10^{-6}$ & $1.52 \myast 10^{-4}$ & $3.46 \myast 10^{-5}$ & $0.99$   \\
 $1024$ & $2.9 \myast 10^{-7}$ & $3.80 \myast 10^{-5}$ & $1.22 \myast 10^{-5}$ & $0.99$   \\
 $2048$ & $7.2 \myast 10^{-8}$ & $9.51 \myast 10^{-6}$ & $4.31 \myast 10^{-6}$ & $0.99$   \\
\hline
\hline
\end{tabular}
\hfill 
\caption{Numerical results for Example \ref{th:example4}, with correction weights}
\label{tab:num5}
\end{table} 
\\
The last column in each table shows that the theory is confirmed in each of the five numerical experiments.
\end{example}
\section{Appendix A: Proof of Lemma \ref{th:omeganinv-props}}
\label{abel-stability}
We now present a proof of \refeq{omeganinv-decay} for the coefficients of the inverse of
the considered generating \powser $ \sum_{\n=0}^{\infty} a_\n \myxi^\n $ which differs from that given by
Eggermont \mynocitea{Eggermont}{81}. Our proof uses Banach algebra theory and may be of independent interest.
\subsection{Special sequence spaces and Banach algebra theory}
We start with the consideration of some sequence spaces in a Banach
algebra framework. For an introduction to Banach algebra theory see, \eg
\mycitea{Rudin}{91}. The following results can be found in Rogozin~\cite{Rogozin[73], Rogozin[76]},
and for completeness they are recalled here.

For a sequence of positive real weights 
$ (\mytau_n)_{n \ge 0} $, consider the following norms, 
\begin{align*}  
\normlinfomeg{a} = \sup_{m \ge 0} \modul{a_m} \mytau_m 
+ \mysum{n=0}{\infty} \modul{a_n},
\qquad
\quad
\normlone{a} = \mysum{n=0}{\infty} \modul{a_n},
\qquad a = (a_n)_{n \ge 0} \subset \koza,
\end{align*}  
and the spaces
\begin{align*}  
\lone & = \inset{a = (a_n)_{n \ge 0} \subset \koza \mid \normlone{a} <  \infty }, 
\qquad \linfomeg = \inset{a = (a_n)_{n \ge 0} \subset \koza \mid \normlinfomeg{a} <  \infty },
\\
\czeromeg & = \inset{a \in \linfomeg \mid a_n \mytau_n  \to 0 \assh n \to \infty}.
\end{align*}  
We obviously have $ \czeromeg \subset \linfomeg \subset \lone $.
By using the canonical identification
$ a(\myt) = \mysumtxt{n=0}{\infty} a_n \myt^n $,
the spaces $ \czeromeg, \linfomeg $ and  $ \lone $ can be considered as function algebras on 
\begin{align*}
\Dr = \inset{\xi \in \koza \mid \modul{\xi} \le 1 },
\end{align*}
the closed unit disc with center 0 and radius $ 1 $.
We are mainly interested in positive weights
$ (\mytau_n)_{n \ge 0} $ which satisfy
$ \mysumtxt{n=0}{\infty} \mytau_n^{-1} < \infty $. In that case,
$ \sup_{m \ge 0} \modul{a_m} \mytau_m $ for 
$ (a_n)_{n \ge 0} \in \linfomeg $ defines a norm on
$ \linfomeg $ which is equivalent to the given norm 
$ \normlinfomeg{\cdot} $. In particular, in the case $ \sigma_0 = 1 $ and
$ \sigma_n = n^\beta $ for $ n = 1,2, \ldots \ (\beta>1) $, 
then $ \linfomeg $ is the space of
sequences $ (a_n)_{n \ge 0} $ satisfying
$ a_n = \Landauno{n^{-\beta}} $ as $ n \to \infty $.
In the sequel we assume
that 
\begin{align} 
\mytau_n \le c \mytau_j, \quad \tfrac{n}{2} \le j \le n, \quad n \ge 0,
\label{eq:omegan_bound}
\end{align} 
holds for some finite constant $ c > 0 $.
We state without proof the following elementary result
(cf.~\mynocitea{Rudin}{91} for part (a) of the proposition, and 
\cite{Rogozin[73], Rogozin[76]} for parts (b) and (c)).
\begin{proposition}
Let
$ \mytau_0, \mytau_1, \ldots $ 
be positive weights satisfying \refeq{omegan_bound}.
\begin{myenumerate}
\item
The space $ \lone $, equipped with convolution
$ (a*b)_n = \mysumtxt{j=0}{n} a_{n-j} b_j, n \ge 0 $, for  $ a,b \in \lone $,
is a commutative complex Banach algebra, with unit $ e = (1,0,0,\ldots) $.
\item
The space $ \linfomeg $ is a subalgebra of $ \lone $, \ie it is closed \wrt
addition, scalar multiplication and convolution.
The norm $ \normlinfomeg{\cdot} $ is complete on $ \linfomeg $
and satisfies
\begin{align} 
\normlinfomeg{a*b} \le (2c+1) \normlinfomeg{a} \cdot \normlinfomeg{b},
\qquad
a, b \in \linfomeg,
\label{eq:normlinfomegr_convolut}
\end{align}  
where $ c $ is taken from estimate \refeq{omegan_bound}.
\item The statements of (b) are also valid for the space
$ \czeromeg $ (instead of $ \linfomeg $), supplied with the norm $ \normlinfomeg{\cdot} $.
\end{myenumerate}
\label{th:linf_czer_algebra}
\end{proposition}
The following proposition is based on the fact that the subalgebra generated by
$ a(\xi) = \xi = (0,1,0,0,\ldots) $ is dense in the space $ \lone $ and in $
\czeromeg $ as well, \ie both spaces are single-generated in fact.
\begin{proposition}[\mycitea{Rogozin}{73}]
Let
$ \mytau_0, \mytau_1, \ldots $
be positive weights satisfying \refeq{omegan_bound}. 
The spaces $ \lone $ and $ \czeromeg $ are inverse-closed, \ie
for each $ a \in \lone $ with $ a(\myt) \neq 0, \, \myt \in \Dr $, one has
 $ [a(\xi)]^{-1} \in \lone $, and 
for each $ a \in \czeromeg $ with $ a(\myt) \neq 0, \myt \in \Dr $, one has
$ [a(\xi)]^{-1} \in \czeromeg $.
\label{th:single_generator_sequence_spaces}
\end{proposition}
For the $ \ell_1 $-case, this is Wiener's theorem, \cf, \eg \mycitea{Rudin}{91}.
The space $ \linfomeg $ is not single-generated but still inverse-closed which will be used in the following. The proof is taken
from \mycitea{Rogozin}{76} and is stated here for completeness.
\begin{proposition}[\mycitea{Rogozin}{76}]
For positive weights $ (\mytau_n)_{n \ge 0} $ satisfying \refeq{omegan_bound}, 
the space $ \linfomeg $ is inverse-closed, \ie
for each $ a \in \linfomeg $ with $ a(\myt) \neq 0 $ for $ \myt \in \Dr $ one has
$ [a(\xi)]^{-1} \in \linfomeg $.
\label{th:linfomeg_inverse-closed}
\end{proposition}
\proof
Consider $ x(\myt) = \mysumtxt{n=0}{\infty} a_n \myt^n  \in \linfomeg $ with 
$ x(\myt) \neq 0 $ for $ \myt \in \Dr $. Then $ x $ is invertible in $ \lone $ (\cf
Proposition \ref{th:single_generator_sequence_spaces}),
\ie $ 1/x(\myt) = \mysumtxt{n=0}{\infty} \ainvn \myt^n  \in \lone $. Let us assume contradictory
that $ 1/x(\myt) \not \in \linfomeg $.
This means that 
$ \limsup_{n \to \infty } \modul{\ainvn[n]} \mytau_{n} = \infty $ and then
\begin{align}  
\tn = \max_{0 \le m \le n} \modul{\ainvn[m]} \mytau_{m} \to \infty \assh 
n \to \infty, 
\label{eq:linfomegr_inverse-closed-1}
\end{align}  
and $ \tn[n+1] \ge \tn > 0 $ for $ n = 0, 1, \ldotsdot $
Let $ \wtilden = \mytau_n / \tn $ for $ n = 0, 1, \ldotsdot $
We have
\begin{align*}  
0 < \wtilden = \frac{\mytau_n}{\tn} \le \frac{\mytau_n}{\tn[j]} \le
c \frac{\mytau_j}{\tn[j]}
= c \wtilden[j],
\quad \tfrac{n}{2} \le j \le n, 
\end{align*} 
so the space $ \czeromegtilde = \inset{a \in \linfomegtil \mid a_n \wtilden  \to 0
\assh n \to \infty} $ with $ \wtilde = (\wtilden)_{n \ge 0} $
is a Banach algebra which is inverse-closed 
(cf.~Propositions \ref{th:linf_czer_algebra} and
\ref{th:single_generator_sequence_spaces}).

By assumption $ \sup_{n \ge 0} \modul{a_n} \mytau_{n} < \infty $ and then
$ \modul{a_n} \wtilden[n] \to 0 $ as $ n \to \infty $.
From Proposition \ref{th:single_generator_sequence_spaces}
it then follows
\begin{align} 
\modul{\ainvn[n]} \wtilden[n] \to 0 \assh n \to \infty.
\label{eq:linfomegr_inverse-closed-2}
\end{align} 
However, it follows from \refeq{linfomegr_inverse-closed-1}
that for some infinite subset $ \Nbf \subset \naza $
we have
\begin{align} 
\tn = \modul{\ainvn[n]} \mytau_{n} \for n \in \Nbf.
\label{eq:linfomegr_inverse-closed-3}
\end{align} 
Otherwise there would exist an $ n_1 \ge 1 $ with
$ 
\tn = \max_{0 \le m \le n} \modul{\ainvn[m]} \mytau_{m} 
> \modul{\ainvn[n]} \mytau_{n} $ for $ n = n_1, n_1+1, \ldots, $
which in fact means 
$ \tn[n-1] = \max_{0 \le m \le n-1} \modul{\ainvn[m]} \mytau_{m} 
> \modul{\ainvn[n]} \mytau_{n} $, 
and then 
$ \tn = \tn[n-1] $ for $  n = n_1, n_1+1, \ldots, $ a contradiction to
\refeq{linfomegr_inverse-closed-1}.
From \refeq{linfomegr_inverse-closed-3}
we then get
\begin{align*} 
\modul{\ainvn[n]} \wtilden[n]  =
\modul{\ainvn[n]} \mytau_n/\tn  = 1,
\quad n \in \Nbf,
\end{align*} 
a contradiction to \refeq{linfomegr_inverse-closed-2}.
\proofend
\subsection{The \powser $\boldsymbol{\sum_{n=0}^\infty (n+1)^{\alp} \myxi^n} $}
Our analysis continues with a special representation of the 
\powser $\sum_{n=0}^\infty (n+1)^{\alp} \myxi^n $, and we will make use of the binomial expansion
\begin{align}
(1-\myxi)^{\beta} = &  \sum_{n=0}^\infty (-1)^n \tbinom{\beta}{n} \myxi^n 
\for \xiunitdisk \qquad (\beta \in \reza),
\label{eq:binomial_a}
\\[-1mm]
& (-1)^n \tbinom{\beta}{n}
= 
\sum_{s=0}^{m-1} d_{\beta,s} n^{-\beta-1-s} +\Landau(n^{-\beta-1-m})
\as n \to \infty,
\label{eq:binomial_b}
\end{align}
with certain real coefficients $ d_{\beta,s} $ for $ s = 0, 1, \ldots, m-1, 
\, m = 0, 1, \dots $,
where $ d_{\beta,0} = 1/\Gamma(-\beta), \beta \neq 0, 1, \ldots $,
 \cf\eg equation (6.1.47) in \myciteb{Abramowitz}{Stegun}{72}.
We need the following result.
\begin{lemma} For $ 0 < \alp < 1 $ we have, with some coefficients $ r_0, r_1, \ldots, $
\begin{align}
& \mfrac{1}{\Gamma(\alpone)}
\sum_{n=0}^\infty (n+1)^{\alp} \myxi^n
=
(1-\myxi)^{\malpone} r(\myxi) \for
\myxi \in \koza, \ \modul{\myxi} < 1,
\label{eq:pow_3_2_b}
\\[-1mm]
&
\with r(\myxi) = \sum_{n=0}^\infty r_n \myxi^n, 
\quad
r(1) = 1, 
\quad r_n = \Landau(n^{\malptwo}) \assh n \to \infty.
\label{eq:r-rep}
 \end{align}
\label{th:pow_3_2}
\end{lemma}
\proof 
We first observe that, for each $ m \ge 0 $, there exist real coefficients
$ c_0, c_1, \ldots, c_{m-1} $ with
\begin{eqnarray}
\mfrac{1}{\Gamma(\alpone)}
\sum_{\n=0}^\infty (\n+1)^{\alp} \myxi^\n
\eq
\sum_{j=0}^{m-1} c_j  (1-\myxi)^{\malpone+j} + s(\myxi) \for
\myxi \in \koza, \ \modul{\myxi} < 1,
\label{eq:pow_3_2}
 \end{eqnarray}
with $ s(\myxi) = \sum_{n=0}^\infty s_n \myxi^n $, where
$ s_n = \Landau(n^{\alpha-m}) $ as $ n \to \infty $, and we have 
$ c_0 = 1 $.
This follows by comparing the coefficients in the Taylor expansion
$ \frac{1}{\gammaalpone} (n+1)^{\alp} = 
\sum_{t=0}^{m-1} e_t n^{\alpone-t} +\Landau(n^{\alpha-m}) $
with the coefficients in the expansions considered in \refeq{binomial_a} and \refeq{binomial_b}.

A reformulation of \refeq{pow_3_2} gives, with $ m = 4 $,
\begin{align*}
\mfrac{1}{\Gamma(\alpone)}
\sum_{n=0}^\infty (n+1)^{\alp} \myxi^n
& =
(1-\myxi)^{\malpone} 
\klala{\sum_{j=0}^{3} c_j  (1-\myxi)^{j} 
+
(1-\myxi)^{\alpone} s(\myxi)} \for
\myxi \in \koza, \ \modul{\myxi} < 1,
\\[-1mm]
& \hspace{10mm}
 \with s(\myxi) = \sum_{n=0}^\infty s_n \myxi^n, \quad s_n = \Landau(n^{\alpmfour})
\as n \to \infty.
\end{align*}
The statement of the lemma now follows from
statement (b) of Proposition \ref{th:linf_czer_algebra}, applied with
$ \sigma_0 = 1 $ and $ \sigma_n = n^{\alptwo} $ for $ n = 1,2, \ldots, $
and from
\refeq{binomial_a}, \refeq{binomial_b} applied with $ \beta = \alpha+1, m = 0 $.
\proofend
\subsection{The main results}
As a consequence of Lemma \ref{th:pow_3_2} we obtain the following representation.
\begin{corollary}
For the quadrature weights $ \myomega_0, \myomega_1, \ldots $
considered in \refeq{omegan-def} we have,
with the \powser $ r $ from  \refeq{pow_3_2_b}, \refeq{r-rep}, 
\begin{align}
\omega(\myxi) = \sum_{n=0}^\infty \myomega_n \myxi^n 
=
(1-\myxi)^{\malp} r(\myxi)
\for \myxi \in \koza, \ \modul{\myxi} < 1.
\label{eq:omega-rep}
 \end{align}
\end{corollary}
\proof
The two \powser $ \sum_{n=0}^\infty (n+1)^{\alp} \myxi^n $ and 
$ \omega(\myxi) = \sum_{n=0}^\infty \myomega_n \myxi^n $
with coefficients as in \refeq{omegan-def}
are obviously related as follows,
\begin{align*}
\sum_{n=0}^\infty \myomega_n \myxi^n
=  
\mfrac{1-\myxi}{\Gamma(\alpone)} \sum_{n=0}^\infty (n+1)^{\alp} \myxi^n.
\end{align*}
The representation \refeq{pow_3_2_b} now implies the statement of the corollary.
\proofend

\bn
Inverting \refeq{omega-rep} immediately gives the \fps representation
\begin{align}
\sum_{n=0}^\infty \aninv \myxi^n = (1-\myxi)^{\alpha} \powinv{r(\myxi)},
\label{eq:omegainv-final-rep}
\end{align}
where
$ \aninv $ denote the coefficients of the \inverse of the \fps
$ \myomega(\myxi) = \sum_{n=0}^\infty \myomegan \myxi^n $, cf.~\refeq{ainv-def}.

In the sequel we examine the asymptotic behavior of the coefficients in the \fps
\begin{align}
\powinv{r(\myxi)} = \sum_{n=0}^\infty \rinv{n} \myxi^n.
\label{eq:rinv-rep}
\end{align}
\begin{lemma}
We have 
$ \rinv{n}
= \Landau(n^{\malptwo}) $ as $ n \to \infty $.
\label{th:rninv-estimate}
\end{lemma}
\proof
It follows from \refeq{r-rep} that the \powser $ r $ considered in \refeq{pow_3_2_b}
satisfies $ r \in \linfomeg $ for the specific choice $ \mytau_0 = 1 $ and  
$ \mytau_n = n^{\alpha+2} $ for $ n \ge 1 $.
In addition we have $ r(\myxi) \neq 0 $ for $ \myxi \in \koza, \modul{\myxi} \le 1 $
(a proof is given below). From Proposition \ref{th:linfomeg_inverse-closed}
we then obtain
$ \rinv{n} =\Landau(n^{\malptwo}) $ as $ n \to \infty $.

It remains to show that 
$ r(\myxi) \neq 0 $ holds for $ \myxi \in \koza, \modul{\myxi} \le 1 $.
For this purpose we consider a reformulation of \refeq{omega-rep},
\begin{align*}
r(\myxi) = (1-\myxi)^{\alpha} \mysum{n=0}\infty \myomegan \myxi^n
\for \myxi \in \koza, \ \modul{\myxi} < 1.
\end{align*}
We have
\begin{align}
\modulbi{\mysum{n=0}\infty \myomegan \myxi^n} \ge 
\mfrac{1}{2\Gamma(\alpone)}
\for \myxi \in \koza, \ \modul{\myxi} < 1,
\label{eq:omegan-lowerbound}
\end{align}
a proof of \refeq{omegan-lowerbound} is presented in the next section. Since
$ r(1) \neq 0 $ and
$ r $ is continuous on $ \insetno{\xi \in \koza \mid \modul{\myxi} \le 1 } $,
\refeq{omegan-lowerbound} then implies
$ r(\myxi) \neq 0 $ for $ \myxi \in \koza, \modul{\myxi} \le 1 $ as desired,
 and thus the statement of the lemma is proved.
\proofend

\bn
We are now in a position to
continue with the verification of the asymptotical behavior \refeq{omeganinv-decay}
for the coefficients of the \powser $ \powinv{\myomega(\myxi)} $.
From the representation 
\refeq{binomial_a}, \refeq{binomial_b} with $ \beta = \alp $
 it follows that the coefficients
in the expansion
$ (1-\myxi)^{\alp} =  \sum_{\n=0}^\infty (-1)^\n \tbinom{\alp}{\n} \myxi^\n $
satisfy 
$ (-1)^\n \binom{\alpha}{\n} = \Landau(\n^{\malpone}) $ as $ \n \to \infty $.
This and Lemma~\ref{th:rninv-estimate}
(which in particular means $ \rinv{n} =  \Landau(n^{\malpone}) $)
and 
part (b) of Proposition \ref{th:linf_czer_algebra}, 
applied with $ \mytau_0 = 1 $ and $ \mytau_n = n^{\alpone} $ for $ n \ge 1 $,
finally results in the desired estimate 
\refeq{omeganinv-decay} for the coefficients of the \powser $ \powinv{\myomega(\myxi)} $.
\subsection{The proof of the lower bound \refeq{omegan-lowerbound}}
To complete our proof of \refeq{omeganinv-decay},
we need to show that
\refeq{omegan-lowerbound} holds. We start with a useful lemma.
\begin{lemma}
The quadrature weights \infseqzerind{\myomega} in \refeq{omegan-def} are positive
and satisfy
$ \sum_{n=0}^{\infty} \myomega_n \allowbreak = \infty $. In addition we have
\begin{align} 
\mfrac{\myomega_{n+1}}{\myomega_n}
> 
\mfrac{\myomega_{n}}{\myomega_{n-1}}
\for n = 1, 2, \ldots \ .
\label{eq:omegan-quo}
\end{align} 
\label{th:omegan-quo}
\end{lemma} 
\proof
It follows immediately from the definition that the 
coefficients \infseqzerind{\myomega} are positive.
The identity $ \sum_{\n=0}^{\infty} \myomegan = \infty $ is obvious, and we next present a proof of the inequality \refeq{omegan-quo}.
Using the notation
\begin{align*} 
f\klasm{x} = x^{\alp} \for x \ge 0 
\end{align*} 
we obtain the following,
\begin{align*}
\mfrac{\myomegan[\myl]}{\myomegan[\myl-1]} 
& =
\mfrac{f\klasm{\myl+1} \minus f\klasm{\myl}}
{f\klasm{\myl} \minus f\klasm{\myl-1}} 
\idstar
\mfrac{ \prim{f}\klasm{t_\myl} }{ \prim{f}\klasm{t_\myl-1} }
=
\klabi{1-\tfrac{1}{t_\myl}}^{1-\alp}
\, =: \, h\klasm{t_\myl} 
\for \myl = 1, 2, \ldots,
\end{align*}
with some real number $ \intervalargo{t_\myl}{\myl}{\myl+1} $.
Here, the identity \klasmsh{*} follows from the generalized mean value theorem.
The function $ h\klasm{s} $ is monotonically increasing for 
$ s > 0 $ which yields estimate
\refeq{omegan-quo}.
\proofendspruch[ of the lemma]

\bn
For results similar to those in Lemma \ref{th:omegan-quo}, see Eggermont~\cite{Eggermont[79], Eggermont[81b]} and
Linz~\cite[Section 10.4]{Linz[85]}.
It follows from Lemma \ref{th:omegan-quo} 
that the conditions of the following lemma 
are satisfied for
$ p_\n = c \myomegan, \ \n = 0, 1, \ldots, $ with $ c > 0 $ arbitrary but
fixed.
\begin{lemma}
[\cf\mycitea{Kaluza}{28}; see also \myciteatwo{Szeg\"{o}}{Szegoe[26]}, \mycitea{Hardy}{48}, and \mycitea{Linz}{85}] 
Let \infseqzerind{p}
be real numbers satisfying 
\begin{align} 
p_\n > 0 \for \n = 0, 1, \ldots,
\qquad
\mfrac{p_{\mys+1}}{p_\mys}
\gge 
\mfrac{p_{\mys}}{p_{\mys-1}}
\for \mys = 1, 2, \ldots \ .
\label{eq:kaluza-cond}
\end{align} 
Then the inverse 
$ \powinv{p\klami{\myxi}} $
of the \powser $ p(\myxi) = \mysumtxt{\n=0}{\infty} p_\n \myxi^\n $ 
 can be written as follows,
\begin{align} 
\powinv{p\klami{\myxi}} = 
c_0 \minus \mysum{\mys=1}{\infty} c_\mys \myxi^\mys, 
\label{eq:kaluza}
\end{align} 
with coefficients \infseqzerind{c} satisfying
$ c_\mys > 0 $ for $ \mys = 0, 1, \ldots \ $.
If moreover
$ \sum_{\mys=0}^{\infty} p_\mys = \infty $ holds and the \fps
$ p(\xi) = \sum_{\mys=0}^{\infty} p_\mys \xi^\mys $ has convergence radius 1,
then we have $ \sum_{\mys=1}^{\infty} c_\mys = c_0 $. 
\label{th:kaluza}
\end{lemma} 
\proof Lemma \ref{th:kaluza} is Theorem 22 on page 68 of 
\mycitea{Hardy}{48}. The proof of $ c_\mys > 0 $ for $ \mys = 0, 1, \ldots \ $ is presented there in full detail, and we do not repeat the steps here.
However, the proof of $ \sum_{\mys=1}^{\infty} c_\mys = c_0 $ is omitted
there, so in the sequel we present some details of this proof.
Condition \refeq{kaluza-cond}
and the assumption on the convergence radius of the power
series $ p(\xi) $ means 
$ p_{\mys+1}/p_\mys \to 1 $ as $ \mys \to \infty $.
The second condition in \refeq{kaluza-cond} then implies
$ 0 < p_{\mys+1} < p_\mys $ for $ \mys = 0,1,\ldots\ $.
From $ c_{\mys} \ge 0 $ for $ \mys = 0,1,\ldots $ we obtain
$
p_{\mys-1} \mysumtxt{\myj=1}{\mys} c_\myj \le 
\mysumtxt{\myj=1}{\mys} p_{\mys-\myj} c_\myj  = p_n c_0 $
for $ \mys = 1,2,\ldots\ $.
The latter identity follows from the representation \refeq{kaluza}.
Thus 
\begin{align*}
\mysum{\myj=1}{\mys} c_\myj \le \mfrac{p_\mys}{p_{\mys-1}} c_0 \le c_0 \for
\mys = 1,2,\ldots \ .
\end{align*}
The latter inequality 
means that 
$ c(\xi) = c_0 -\mysumtxt{\myj=1}{\infty} c_\myj \xi^\myj $ is absolutely convergent on the closed unit disc $ \inset{ \xi \in \koza \mid \modul{\xi} \le 1 } $ and hence is continuous on this set. This finally gives
\begin{align*}
0 =
\lim_{0 < x \to 1} \mfrac{1}{\mysumtxt{\myj=0}{\infty} p_\myj x^\myj }
= 
c_0 - \lim_{0 < x \to 1} \mysum{\myj=1}{\infty} c_\myj x^\myj 
= c_0 - \mysum{\myj=1}{\infty} c_\myj.
\end{align*}
This completes the proof of the lemma.
\proofend

\bn
The following lemma is closely related to results in \myciteatwo{Erd\H{o}s, Feller and
Pollard}{Erdoes_Feller_Pollard[49]}. A detailed proof can be found in
\mynocitea{Plato}{12}.
\begin{lemma}
Let \infseqind{c} be a sequence of real numbers satisfying 
$ c_\myl > 0 $ for $ \myl = 1, \cdott 2, \ldots, $ and 
$ \sum_{\myl=1}^{\infty} c_\myl = \tfrac{1}{2} $.
Then the \powser
$ q\klasm{\myxi} = \tfrac{1}{2} - \sum_{\myl=1}^{\infty} c_\myl \myxi^\myl $ satisfies
$ \modul{q\klasm{\myxi}} < 1 $ for each complex number $ \myxi $ with 
$  \modul{\myxi} \le 1 $.
\label{th:powsermod}
\end{lemma} 
We are now in a position to present a proof of the lower bound \refeq{omegan-lowerbound}.
In fact, from Lemma \ref{th:omegan-quo} it follows that the coefficients
of the \powser 
$ p\klasm{\myxi} = 2 \Gamma(\alpone)\myomega\klasm{\myxi} $
with $ \omega(\xi) $ as in \refeq{omega-rep}
satisfy the conditions of Lemma \ref{th:kaluza}, and in addition $ p_0 = 2 $ holds.
This implies that the coefficients of the \powser 
\begin{align*}
\frac{1}{2\Gamma(\alpone)\myomega\klasm{\myxi}} 
=
c_0 \minus \mysum{\n=1}{\infty} c_\myl \myxi^\n
\end{align*}
satisfy
$ c_\n > 0 $ for $ \n = 0, 1, \ldots $ and
$ \sum_{\n=1}^{\infty} c_\n = c_0 = 1/2 $.
Lemma \ref{th:powsermod}
then implies that 
$ 2 \Gamma(\alpone) \modul{\myomega\klasm{\myxi}} \ge 1 $
and thus 
$ \modul{\myomega\klasm\xi} \ge \tfrac{1}{2 \Gamma(\alpone)} $
for $ \myxi \in \koza, \modul{\myxi} < 1 $. 
This is the desired estimate \refeq{omegan-lowerbound} needed in the proof of 
Lemma \ref{th:rninv-estimate}.
\section{Appendix B: Proof of Theorem \ref{th:main-midpoint}}
\label{appendix_b}
1.~We apply the representations \refeq{midpoint-rule} and
\refeq{midpoint-rule-error-def} with
$ \myfun=\myfun_n $, where
\begin{align*}
\myfun_n(y) =  k\klasm{\xn,y} u\klasm{y}, \quad 0 \le y \le \xn,
\end{align*}
and scheme \refeq{midpoint-noise-rule} imply the following,
\begin{align}
\halp \mysum{j=1}{n} \myomega_{n-j} \cdott k\klasm{\xn,\xn[j-1/2]} \cdott
\enndelta[j-1/2] 
= \enn{\myfun_n}{\xn}
+ \fndelta - f(\xn)
\for n = \myseqq{1}{2}{\nmax},
\label{eq:main-midpoint-a}
\end{align}
where
\begin{align*}
\enndelta & = \undelta[j-1/2] - u\klasm{x_{j-1/2}}, \quad j=\myseqq{1}{2}{\nmax}.
\end{align*}
2.~We next consider 
a \matvecform of \refeq{main-midpoint-a}. 
As a preparation we consider the matrix $ \Ah \in \myrnn[\nmax] $ is given by
\feldstretch{2.0}
\begin{align*}
\Ah \eq 
\left(
\begin{array}{@{\quadti}c@{\quadti}c@{\quadti}c@{\quadti}c@{\quadti}c@{\quadti}}
\myomega_0 \myk{1}{1/2} & 0 & \cdots & \cdots & 0 \\
\myomega_1 \myk{2}{1/2} & \myomega_0 \myk{2}{3/2} & 
\ddots & & 0 \\
\vdots & \myomega_1 \cdott \myk{3}{3/2} & \ddots  & \ddots &  \vdots \\
\vdots & & \ddots  & \ddots & 0 \\
\myomega_{\nmax-1} \myk{\nmax}{1} & \cdots & \cdots & 
\myomega_1 \myk{\nmax}{\nmax-3/2} &
\myomega_0 \myk{\nmax}{\nmax-1/2}
\end{array} \right)
\end{align*}
with the notation
\begin{eqnarray*}
\myk{n}{j-1/2} \eq k\klasm{\xn,\xn[j-1/2]}
\for 1 \le j \le n \le \nmax.
\end{eqnarray*}
Additionally we consider the vectors 
\begin{align}
\Ehdelta &= \kla{\enndelta[j-1/2]}_{j=1,2,\ldots,\nmax},
\quad
\rh = \kla{\enn{\myfun_n}{\xn}}_{n=1,2,\ldots,\nmax},
\quad
\Fh = 
\kla{\fndelta- f(\xn[n])}_{n=1,2,\ldots,\nmax}.  
\label{eq:main-midpoint-b-1}
\end{align}
Using these notations, the linear system of equations \refeq{main-midpoint-a}
can be written as 
\begin{align}
\halp \Ah \Ehdelta  = \rh \plus \Fh, \with \maxnorm{\Fh} \le \delta,
\label{eq:main-midpoint-b-2}
\end{align}
where $ \maxnorm{\cdot} $ denotes the maximum norm on $ \reza^\N $.
In addition, occasionally we consider a modified error equation which can easily be derived from \refeq{main-midpoint-b-2} by applying the matrix $ \Dh $ to
both sides of that equation:
\begin{align}
\halp \Dh \Ah \Ehdelta  = \Dh \rh +  \Dh \Fh.
\label{eq:main-midpoint-b-3}
\end{align}
This technique is a discrete analogue of fractional differentiation.

3.~For a further treatment of the identity \refeq{main-midpoint-b-2} and its variant
\refeq{main-midpoint-b-3},
we
now show
\begin{align}
\maxnorm{\Dh } = \Landausm{1},
\qquad
\maxnorm{\klasm{\Dh \Ah}^{-1}} 
= \Landauno{1},
\qquad
\maxnorm{\Ah^{-1} } = \Landausm{1} \as h \to 0,
\label{eq:main-midpoint-c}
\end{align}
where the matrix $ \Dh \in \myrnn[\nmax] $ given by
\feldstretch{2.1}
\begin{align}
\Dh = 
\left(
\begin{array}{c@{\quadsm}c@{\quadsm}c@{\quadsm}c@{\quadsm}c}
\aninv[0] & 0 & \cdots & \cdots & 0 \\
\aninv[1] & \aninv[0] & 0 & & 0 \\ 
\aninv[2] & \ddots & \ddots & \ddots & \vdots \\ 
\vdots & \ddots & \ddots & \ddots & 0 \\
\aninv[\nmax-1] & \cdots & \cdots  &
\aninv[1] & \aninv[0] 
\end{array} \right) \cdott ,
\label{eq:Dh-def}
\end{align}
and $ \maxnorm{\cdot} $ denotes the matrix norm induced by the maximum vector norm on $ \reza^\N $. 
In fact, the estimate $ \maxnorm{\Dh } = \Landausm{1} $ as $ h \to 0 $
follows immediately from the decay of the coefficients of the \inverse of the \genfunc $ \myomega $,
\cf estimate \refeq{omeganinv-decay}. 
For the proof of the second statement in \refeq{main-midpoint-c} we use the fact that 
the matrix $ \Dh \Ah  $ can be written in the form
$ \Dh \Ah = I_h  \plus \cdott K_h $, where 
$ I_h \in \myrnn[\N] $ denotes the identity matrix,  
and $ K_h = \kla{k_{h,\n,j}} \in \myrnn[\N]  $
denotes some lower triangular matrix which satisfies
$ \max_{1 \le j \le \n \le \nmax} {\modul{k_{h,\n,j}}} = \Landausm{h} $ as $ h \to 0 $,
\cf the proof of Lemma 4.2 in \mycitea{Eggermont}{81} for more details. We only
note that here it is taken into account that the kernel function is uniformly
Lipschitz continuous with respect to the first variable, cf.~part 
\ref{item:assump-k-smooth} of Assumption \ref{th:midpoint-assump}.
This representation of $\Dh \Ah $ and the discrete version of Gronwall's inequality 
now yields $ \maxnorm{\klasm{\Dh \Ah}^{-1}} = \Landauno{1} $ as $ h \to 0 $.
The third estimate in \refeq{main-midpoint-c} follows immediately from the other two 
estimates considered in
\refeq{main-midpoint-c}.

4.~In view of \refeq{main-midpoint-b-2}--\refeq{main-midpoint-c}, it remains to
take a closer look at the representations of the quadrature error 
considered in Lemma \ref{th:midpoint-error}.
We consider different situations for $ \gamma $
and constantly make use of the fact that,  for some finite constant $ L \ge 0 $, 
we have 
\begin{align}
\myfunn \in \HL{\gamma}{\xn} \for \n = 1,2, \ldots, \N, 
\label{eq:myfunn_hoelder}
\end{align}
\cf
Assumption \ref{th:midpoint-assump}. 
\begin{myenumerate_roman}
\item In the case $ \gamma \le 1 $ we proceed in two different ways. The first one turns out to be useful for the case $ \alp \le \tfrac{1}{2} $, while the other one uses partial summation and is useful for the case $ \alp \ge \tfrac{1}{2} $.
\begin{mylist}
\item
Our first approach proceeds with \refeq{main-midpoint-b-2}, and we assume $ \alp < \gamma \le 1 $ in this case. We then 
easily obtain, cf.~\refeq{midpoint_error_0}, \refeq{myfunn_hoelder},
\begin{align*}
\maxnorm{\rh} = \max_{1 \le \n \le \nmax} \modul{\enn{\myfunn}{\xn}} = \Landauno{\mydeltax^{\gamma}} \as h \to 0,
\end{align*}
and then, cf.~\refeq{main-midpoint-b-2} and \refeq{main-midpoint-c},
$ \maxnorm{\Ehdelta}
 = \Landauno{\mydeltax^{\malp}(\mydeltax^{\gamma} + \delta)}
= \Landauno{\mydeltax^{\gamma-\alpha} + \tfrac{\delta}{\mydeltax^\alp}} $.

\item
In our second approach we would like to proceed with \refeq{main-midpoint-b-3}, and  
we need to consider the vector $ \Dh\rh \in \reza^\N $ in more detail. 
For this purpose we assume that $ 1-\alp < \gamma \le 1 $ holds, and
we introduce the notation
\begin{align*} 
\rnh = \enn{\myfun_n}{\xn}, \quad n=\myseqq{1}{2}{\nmax}.
\end{align*}
Partial summation, applied to the \nth entry of
$ \Dh\rh $, gives
\begin{align} 
(\Dh\rh)_n = \mysum{\jod=1}{\n} \aninv[\n-\jod] \rnh[\jod]
=
\mybeta_n \rnh[1] 
+
\mysum{\ell=1}{\n-1} \mybeta_{n-\ell}
\kla{\rnh[\ell+1]- \rnh[\ell]},
\label{eq:partial-sum} 
\end{align} 
where
\begin{align}
\mybeta_\n = \mysumtxt{\ell=0}{\n-1} \aninv[\ell] \for n = 1, 2, \ldots \ .
\label{eq:betan-def} 
\end{align} 
From Lemma \ref{th:omeganinv-props} it easily follows that 
$ \mybeta_\n \ge 0 $ for $ n = 1, 2, \ldots $\ . In addition, we have
\begin{align}
\mybeta_\n = \Landauno{\n^{\malp}} \as n \to \infty,
\label{eq:betan-decay} 
\end{align} 
and thus
\begin{align}
\mysum{\ell=1}{\n-1} \mybeta_{\ell}
= \Landauno{\N^{1-\alp}} 
= \Landauno{h^{\alp-1}} \as h \to 0 
\label{eq:betan-sum} 
\end{align} 
uniformly for $ n = 1,2,\ldots, N $.
H\"older continuity \refeq{myfunn_hoelder} implies
\begin{align*} 
\modul{\rnh[1]} = \modul{\enn{\myfunn[1]}{\xn[1]}}
\le 
\mfrac{ L h^\gamma}{\Gamma(\alp)} \ints{0}{h} { (h-y)^{\alp-1} }{ dy }
= \Landauno{h^{\gamma+\alp}},
\end{align*} 
and we next consider the differences $ \rnh[\ell+1]- \rnh[\ell] $ in more detail.
For this purpose we introduce short notation for the interpolation error,
\begin{align*}
\mydelta_\n(y) = \myfunn(y) - \ph \myfunn (y) \for 0 \le \vary \le \xn,
\quad n = 1,2,\ldots, \N.
\end{align*}
We then have
\begin{align*} 
\rnh[\ell+1] - \rnh[\ell] 
& =
\tfrac{1}{\Gamma(\alp)} \klala{ \ints{0}{\xn[\ell+1]} 
{ (\xn[\ell+1]-\vary)^{\alp-1} \mydelta_{\ell+1}(y) }{ dy }
- \ints{0}{\xn[\ell]} 
{ (\xn[\ell]-\vary)^{\alp-1} \mydelta_\ell(y) }{ dy }
}
\\
&= 
\tfrac{1}{\Gamma(\alp)} \ints{\xn[\ell]}{\xn[\ell+1]} 
{ (\xn[\ell+1]-\vary)^{\alp-1} \mydelta_{\ell+1}(y) }{ dy }
+
\tfrac{1}{\Gamma(\alp)} \ints{0}{\xn[\ell]} 
{ (\xn[\ell+1]-\vary)^{\alp-1} \kla{\mydelta_{\ell+1}-\mydelta_{\ell}}(y) }{ dy }
\\
& \qquad +
\tfrac{1}{\Gamma(\alp)} \ints{0}{\xn[\ell]} 
{ \kla{(\xn[\ell+1]-\vary)^{\alp-1} -  (\xn[\ell]-\vary)^{\alp-1}} \mydelta_{\ell}(y) }{ dy }
=: s_1 + s_2 + s_3.
\end{align*} 
We have $ s_1 = \Landauno{h^{\gamma + \alp }} $ which easily follows from
$ \sup_{0 \le \vary \le \xn[\ell+1]} \modul{\mydelta_{\ell+1}(\vary)} = \Landauno{h^\gamma} $. Moreover, a first order Taylor expansions of the kernel $ k $ with respect to the first variable at the grid point $ \xjb \ (1 \le j \le \ell+1 ) $ gives,
for $ \xjm \le \vary \le  \xj $,
\begin{align*} 
\kla{\mydelta_{\ell+1}-\mydelta_{\ell}}(y)
&=
k(\xn[\ell+1],\vary) u(y) - k(\xn[\ell+1],\xjb) u(\xjb)
\\ 
& \qquad  - \big\{
k(\xn[\ell],\vary) u(y) - k(\xn[\ell],\xjb) u(\xjb) \big\}
\\
&=
\big(\mfrac{\partial k}{\partial x}(\xn[\ell],\vary) h + \Landauno{h^2}\big) u(y)
- \big(\mfrac{\partial k}{\partial x}(\xn[\ell],\xjb) h + \Landauno{h^2}\big) u(\xjb)
\\
&=
h \big(\mfrac{\partial k}{\partial x}(\xn[\ell],\vary) u(y) -
\mfrac{\partial k}{\partial x}(\xn[\ell],\xjb) u(\xjb) \big)
+ \Landauno{h^2}
=
\Landauno{h^{\gamma + 1}},
\end{align*} 
and this implies
$ s_2 = \Landauno{h^{\gamma + 1}} $.
Finally, 
\begin{align*} 
\modul{s_3} 
&\le
\mfrac{L}{\Gamma(\alpha)} h^\gamma \ints{0}{\xn[\ell]} 
{ (\xn[\ell]-\vary)^{\alp-1} -  (\xn[\ell+1]-\vary)^{\alp-1}}{ dy }
=
\mfrac{L}{\Gamma(\alpone)} h^{\gamma + \alp}  \kla{ 1 + \ell^\alp- (\ell + 1)^\alp}
= \Landauno{h^{\gamalp}}.
\end{align*} 
\marginpar{Check}
Summation gives
$ s_1 + s_2 + s_3 = \Landauno{h^{\gamma + \alp}} $,
and \refeq{partial-sum} finally results in (see also \refeq{betan-sum})
\begin{align*} 
(\Dh\rh)_n 
&=
\Landauno{h^{\gamma + \alp} + h^{\alp-1} h^{\gamma + \alp}}
=
\Landauno{h^{\gamma + 2\alp -1}}
\end{align*} 
uniformly for $ n = \myseqq{1}{2}{\nmax} $. We note that this estimate is useful for
$ \alpha \ge \tfrac{1}{2} $ only. 
We are now in a position to proceed with \refeq{main-midpoint-b-3}:
\begin{align*}
\maxnorm{\Ehdelta}  =
\Landaubi{\hmalp \maxnorm{\Dh\rh} + \mfrac{\delta}{\halp}}
= \Landauno{h^{\gamma + \alp -1} + \mfrac{\delta}{\halp}}
\as (h,\delta) \to 0,
\end{align*}
where also \refeq{main-midpoint-c} has been used. This gives the desired result.
\end{mylist}

\item
\label{it:main-proof-ii} 
We now proceed 
with the case $ 1 < \gamma \le 2 $. Preparatory results are given in the present item \ref{it:main-proof-ii}, and in item \ref{it:main-proof-iii} the final steps will be done.
Representation~\refeq{midpoint_error_1} of the integration error gives
\begin{align*}
\enn{\myfunn}{\xn}
&=
\halpone \rhna + \rhnb, \with
\rhna = \mysum{\jod=1}{\n} \tau_{\n-\jod} \prim{\myfunn}\kla{\xjb},
\quad
\rhnb = \Ialpkla{\myfunn- \qh \myfunn}{\xn},
\end{align*}
for $ n=1,2,\ldots, \N $,
or, in vector notation
(for the definition of $ \rh $ see \refeq{main-midpoint-b-1})
\begin{align}
\rh
& = 
\halpone \rha + \rhb,
\with \rha = (\rhna)_{n=1,2,\ldots, \N},
\quad
\rhb = (\rhnb)_{n=1,2,\ldots, \N}.
\label{eq:main-midpoint-g}
\end{align}
In view of
\refeq{main-midpoint-b-2} and \refeq{main-midpoint-b-3},
we need to consider the four vectors 
$ \rha, \Dh\rha, \rhb $ and $ \Dh\rhb \in \reza^\N $ in more detail.
\begin{mylist}
\item  
From the summability of the coefficients $ \taun $, \cf \refeq{taul_asymp}, it immediately follows that 
$ \maxnorm{\rha}  = \Landauno{1} $ as $ h \to 0 $.

\item
In the case $ \gamma > 2 - \alpha $ and $ u(0) = \prim{u}(0) = 0 $,
it turns out to be useful to consider the vector $ \Dh\rha $. Partial summation related to the \nth entry of
$ \Dh\rha $ gives
\begin{align} 
(\Dh\rha)_n \eq \mysum{\ell=1}{\n} \aninv[\n-\ell] \rhna[\ell]
\eq
\mybeta_n \rhna[1]
+
\mysum{\ell=1}{\n-1} \mybeta_{n-\ell}
\kla{\rhna[\ell+1] - \rhna[\ell]},
\label{eq:partialsum-2}
\end{align} 
with $ \mybeta_\n $ given by \refeq{betan-def}.
The smoothness property \refeq{myfunn_hoelder}, 
the assumption $ u(0) = \prim{u}(0) = 0 $
and the boundedness
$ \mybeta_n = \Landauno{1} $, \cf \refeq{betan-decay},
imply that 
$ \mybeta_n \rhna[1] = 
\mybeta_n \tau_0  \prim{\myfun}_1\kla{\varx_{1/2}} = \Landauno{h^{\gamma-1}} $.
In addition,
\begin{align*} 
\rhna[\ell+1] - \rhna[\ell]
& =
\mysum{\jod=1}{\ell+1} \tau_{\ell+1-\jod} \prim{\myfunn[\ell+1]}\kla{\xjb}
-
\mysum{\jod=1}{\ell} \tau_{\ell-\jod} \prim{\myfunn[\ell]}\kla{\xjb} 
\\[-1mm]
& =
\tau_{\ell} \prim{\myfunn[\ell+1]}\kla{x_{1/2}}
+ \mysum{\jod=1}{\ell} \tau_{\ell-\jod} 
\klabi{\prim{\myfunn[\ell+1]}\kla{\xjf} - \prim{\myfunn[\ell]}\kla{\xjb} }
= \Landauno{h^{\gamma-1}}
\end{align*}
uniformly for $ \ell = \myseqq{1}{2}{\N-1} $.   
The considered partial summation \refeq{partialsum-2}
thus finally results in (see also \refeq{betan-sum})
\begin{align} 
\maxnorm{\Dh\rha} = 
\Landauno{h^{\gamma - 1}} + \Landauno{h^{\alp-1+\gamma - 1}}
= \Landauno{h^{\gamalp -2}}.
\label{eq:partialsum-2a}
\end{align} 
\item
It follows from \refeq{interpol-error-2} that
$ \maxnorm{\rhb}  = \Landauno{h^\gamma} $ as $ h \to 0 $.
This estimate will be useful in the case $ \alp \le \tfrac{1}{2} $.

\item
We next consider the vector 
$ \Dh \rhb $ in more detail.
Partial summation applied to the \nth entry of
$ \Dh\rhb $ gives
\begin{align} 
(\Dh\rhb)_n \eq \mysum{\ell=1}{\n} \aninv[\n-\ell] \rhnb[\ell]
=
\mybeta_n \rhnb[1]
+
\mysum{\ell=1}{\n-1} \mybeta_{n-\ell}
\kla{\rhnb[\ell+1] - \rhnb[\ell]}.
\label{eq:partialsum-3}
\end{align} 
We have
\begin{align*} 
\rhnb[1] = \Landauno{h^{\gamma+\alp}},
\qquad 
\rhnb[\ell+1] - \rhnb[\ell]
= \Landauno{h^{\gamma+\alp}},
\end{align*} 
uniformly for $ \ell = \myseqq{1}{2}{\N-1} $. This in fact is verified similarly as in the second item of part 4(i) of this proof, this time with second order Taylor expansions of the kernel $ k $ as well as first order Taylor expansions of $ \tfrac{\partial k}{\partial y} $ with respect to the first variable, respectively.
We omit the simple but tedious computations.
This gives
\begin{align} 
\maxnorm{\Dh\rhb} = 
\Landauno{h^{\gamma + \alp}} + \Landauno{h^{\alp-1+\gamma + \alp}}
= \Landauno{h^{\gamma + 2 \alp-1}}.
\label{eq:partialsum-3a}
\end{align} 
This estimate will be useful in the case $ \alp \ge \tfrac{1}{2} $.
\end{mylist}
It should be noticed that the second of the four considered items is the only one where the initial condition $ u(0) = \prim{u}(0) = 0 $ is needed. 

\item
\label{it:main-proof-iii} 
We continue with the consideration of the case $ 1 < \gamma \le 2 $. The results from
\ref{it:main-proof-ii} 
allow us to proceed with 
\refeq{main-midpoint-b-2}, \refeq{main-midpoint-b-3}.
\begin{mylist}
\item
We first consider the case $ \alp \le \tfrac{1}{2}, \cdott  1 < \gamma \le \alpone $. 
The consistency error representations in item (ii) of the present proof
yield 
$ \maxnorm{\rh} = \max_{1 \le \n \le \nmax} \modul{\enn{\myfunn}{\xn}} 
= \Landauno{h^{\alpone} \maxnorm{\rha} + \maxnorm{\rhb}}
= \Landauno{h^{\alpone} + h^{\gamma}} = \Landauno{h^{\gamma}} $.
From the error equation \refeq{main-midpoint-b-2} it then follows
$ \maxnorm{\Ehdelta}
= \Landauno{\mydeltax^{\malp}(\mydeltax^{\gamma} + \delta)}
= \Landauno{\mydeltax^{\gamma-\alpha} + \delta/\mydeltax^\alp} $.

\item
We next consider the case $ \alp \ge \tfrac{1}{2}, \cdott  1 < \gamma \le 2-\alp $. 
The integration error estimates obtained in item (ii)
yield 
$ \maxnorm{\Dh \rh}
= \Landauno{h^{\alpone} \maxnorm{\rha} + \maxnorm{\Dh \rhb}}
= \Landauno{h^{\alpone} + h^{\gamma+ 2\alp-1}} = \Landauno{h^{\gamma+ 2\alp-1}} $,
where the first identity in \refeq{main-midpoint-c} has been applied.
From the error equation \refeq{main-midpoint-b-3} it then follows
$ \maxnorm{\Ehdelta}
= \Landauno{\mydeltax^{\malp}(\mydeltax^{\gamma+ 2\alp-1} + \delta)}
= \Landauno{\mydeltax^{\gamma-1+\alp} + \delta/\mydeltax^\alp} $.

\item
Finally we consider the case $ 2-\alp < \gamma \le 2 $ and $ u(0) = \prim{u}(0) = 0 $.
The consistency error estimates in item (ii)
yield 
$ \maxnorm{\Dh \rh}  
= \Landauno{h^{\alpone} \maxnorm{\Dh \rha} + \maxnorm{\Dh \rhb}}
= \Landauno{h^{\gamma+ 2\alp-1}} $.
From the error equation~\refeq{main-midpoint-b-3} it then follows
$ \maxnorm{\Ehdelta}
= \Landauno{\mydeltax^{\malp}(\mydeltax^{\gamma+ 2\alp-1} + \delta)}
= \Landauno{\mydeltax^{\gamma-1+\alp} + \delta/\mydeltax^\alp} $.
\end{mylist}
This completes the proof of the theorem.
\end{myenumerate_roman}
\section{Conclusions}
In the present paper we have considered
the \repmidrule for the regularization of
weakly singular Volterra integral
equations of the first kind
with perturbed given \rhss. 
The applied techniques are closely related to those used 
in \mycitea{Eggermont}{79}.
The presented results include intermediate smoothness degrees of the solution of the integral equation in terms of H\"older continuity.
In addition we have given a new proof of the stability estimate for the \inverse of the
generating sequence, \cf\refeq{omeganinv-decay}, which may be of independent interest.
Another topic is the use of correction starting weights to get rid of initial conditions on the solution. Results of some numerical experiments are also given. 
\bibliography{../datenbanken/standard,../datenbanken/volterra,../datenbanken/numa,../datenbanken/illposed,../datenbanken/fuan}
\end{document}